\documentclass{scrartcl}

\usepackage{amsmath}
\usepackage{amsfonts}
\usepackage{amssymb}
\usepackage{amsthm}
\usepackage{dsfont}
\usepackage{enumerate}
\usepackage{url}
\usepackage[english]{babel}
\usepackage[runin]{abstract}

\usepackage{graphicx}
\usepackage{tikz}
\usepackage{color}
\usepackage[format=plain,labelfont=bf,textfont=it]{caption}
\usepackage[numbers]{natbib}
\usepackage{mathrsfs}
\usepackage{stackrel}

\newtheorem{theorem}{Theorem}
\newtheorem{lemma}[theorem]{Lemma}
\newtheorem{corollary}[theorem]{Corollary}
\newtheorem{proposition}[theorem]{Proposition}
\theoremstyle{definition}
\newtheorem{definition}{Definition}
\newtheorem*{pseudodef}{Pseudodefinition}
\newtheorem{example}{Example}
\newtheorem*{remark}{Remark}

\newcommand{\cC}{\mathcal{C}}
\newcommand{\cE}{\mathcal{E}}

\newcommand{\cH}{\mathcal{H}}
\newcommand{\cL}{\mathcal{L}}
\newcommand{\cM}{\mathcal{M}}
\newcommand{\cT}{\mathcal{T}}
\renewcommand{\d}{\mathrm{d}}
\newcommand{\D}{\mathrm{D}}
\newcommand{\der}[2]{\frac{\partial #1}{\partial #2}}
\newcommand{\disc}{\mathrm{disc}}

\newcommand{\inner}[2]{\left\langle #1 \,, #2 \right\rangle}
\newcommand{\binner}[2]{\big\langle #1 \,, #2 \big\rangle}
\newcommand{\mesh}{\mathrm{mesh}}
\renewcommand{\mod}{\mathrm{mod}}
\renewcommand{\O}{\mathcal{O}}
\newcommand{\R}{\mathbb{R}}

\title{
Modified Equations for \\ Variational Integrators}
\author{Mats Vermeeren}
\date{\normalsize \textit{Institut f\"ur Mathematik, MA 7-2, Technische Universit\"at Berlin, \\
Str.\@ des 17.\@ Juni 136, 10623 Berlin, Germany \\
E-mail:} \url{vermeeren@math.tu-berlin.de}}

\begin{document}

\maketitle
\thispagestyle{empty}
\enlargethispage{1cm}

\abslabeldelim{.}
\setlength{\abstitleskip}{-\absparindent}
\begin{abstract}
It is well-known that if a symplectic integrator is applied to a Hamiltonian system, then the modified equation, whose solutions interpolate the numerical solutions, is again Hamiltonian. We investigate this property from the variational side. We present a technique to construct a Lagrangian for the modified equation from the discrete Lagrangian of a variational integrator.
\end{abstract}

\section{Introduction}

The key to explaining the long-time behavior of symplectic integrators is backward error analysis, the study of the modified equation whose solutions  interpolate the discrete solutions. It is a well-known and essential fact that if a symplectic integrator is applied to a Hamiltonian equation, then the resulting modified equation is Hamiltonian as well. This strongly suggest that when a variational integrator is applied to a Lagrangian system, the resulting modified equation is Lagrangian as well. In this paper we investigate whether that is indeed the case. 

We will introduce a method to construct modified Lagrangians directly from the discrete Lagrangian of the variational integrator. These modified Lagrangians produce the modified equation up to an error of arbitrarily high order in the step size. Our method is similar to an approach taken by \citet*{oliver2012new}, who discuss the analogous problem for a variational semi-discretization of the semi-linear wave equation.

First we briefly review the essentials of variational integrators. Then, in Section \ref{sec-modeqn}, we discuss the concept of modified equations for first and second order difference equations. In Section \ref{sec-modlag} we present the construction of a Lagrangian for the modified equation. It consists of many steps and requires the introduction of some analytic concepts. The most important of those are \emph{meshed variational problems}, where an extremizer is sought in a class of curves that may be nondifferentiable at some points, and \emph{$k$-critical families of curves}, which are families of curves that almost extremize the action. Finally, in Section \ref{sec-ex}, we clarify our approach with some examples.

\section{Variational integrators}

In this section we give a concise introduction to variational integrators, inspired on \citet*[Section VI.6]{hairer2006geometric}. For a detailed overview of the concept and an extensive bibliography, we refer to \citet*{marsden2001discrete}.

A continuous \emph{Lagrangian} or \emph{variational} system on the Euclidean space $\R^N$ is described by a smooth function $\cL: T\R^N \cong \R^N \times \R^N \rightarrow \R$ and the corresponding \emph{action integral}
\begin{equation}\label{cont-act}
 S(x) := \int_a^b \cL(x(t),\dot{x}(t)) \,\d t.
\end{equation}
A smooth curve $x:[a,b] \rightarrow \R^N: t \mapsto (x_1(t), \ldots, x_N(t))$ is a solution of the system if and only if it is a critical point of the action $S$ in the set of all smooth curves with the same endpoints $x(a)$ and $x(b)$. Formally, this condition can be written as
\begin{align}
 0 = \delta S(x) = \int_a^b \delta \cL(x(t),\dot{x}(t)) \,\d t
&= \int_a^b \sum_{i=1}^N \left( \der{\cL}{x_i} \delta x_i + \der{\cL}{\dot{x}_i} \delta \dot{x}_i \right) \d t \notag \\
&= \int_a^b \sum_{i=1}^N \left( \der{\cL}{x_i} - \frac{\d}{\d t} \der{\cL}{\dot{x}_i} \right) \delta x_i \,\d t. \label{cont-varact}
\end{align}
When integrating by parts to obtain the last equality we could ignore the boundary term because the boundary values of the curve are fixed, hence $\delta x(a) = \delta x(b) = 0$. Since Equation \eqref{cont-varact} holds for any such variation $\delta x$, the criticality of the action is characterized by the conditions
\[ \der{\cL}{x_i} - \frac{\d}{\d t} \der{\cL}{\dot{x}_i} = 0 \qquad \text{ for } i = 1, \ldots, N, \]
which are known as the \emph{Euler-Lagrange equations}. We will usually write them as a single vector-valued equation,
\begin{equation}\label{cont-EL}
\der{\cL}{x} - \frac{\d}{\d t} \der{\cL}{\dot{x}} = 0. 
\end{equation}
In general this is a second order differential equation. We will assume that the Lagrangian is \emph{regular}, i.e.\@ $\det \der{^2 \cL}{\dot{x}^2} \neq 0$. Then the Euler-Lagrange equation can always be solved for $\ddot{x}$.

One approach to discretizing the Euler-Lagrange equation \eqref{cont-EL} is to discretize the action integral \eqref{cont-act} and to consider discrete curves that are a critical points of this discrete action. Usually, one looks for a discrete Lagrange function $L_\disc: \R^N \times \R^N \times \R_{>0} \rightarrow \R$ and defines the discrete action as
\[ S_\disc((x_j)_j,h) := \sum_{j=1}^n h L_\disc(x_{j-1},x_j,h). \]
A discrete curve $x = (x_0, \ldots, x_n)$ is a critical point of $S_\disc$ in the set of all discrete curves with the same endpoints $x_0$ and $x_n$ for some fixed step size $h$ if and only if it satisfies the \emph{discrete Euler-Lagrange equation}
\begin{equation}\label{disc-EL}
\D_2 L_\disc(x_{j-1},x_j,h) + \D_1 L_\disc(x_j,x_{j+1},h) = 0 \qquad \text{for } j=1,\ldots,n-1,
\end{equation} 
where $\D_1 L_\disc$ and $\D_2 L_\disc$ denote the partial derivatives of $L_\disc$.

A discrete Lagrangian can be scaled by any nonzero $h$-dependent factor without affecting the dynamics. The following concept provides a natural scaling.

\begin{definition}\label{def-consistency}
\begin{enumerate}[$(a)$]
\item A smooth function $\Phi: \left( \R^N  \right)^2 \times \R_{>0} \rightarrow \R$ is a \emph{consistent discretization} of a smooth function $g: T\R^N \rightarrow \R$ if there exist the functions $g_i: \left(\R^N\right)^{n_i} \rightarrow \R$ such that for any smooth curve $x$ there holds
\[
\Phi(x(t),x(t+h),h) = g(x(t),\dot{x}(t)) + \sum_{i=1}^\infty h^i g_i[x(t)].
\]
If $x$ is not analytic this should be interpreted as an asymptotic expansion. In particular, this implies that
\[ \Phi(x(t),x(t+h),h) = g(x(t),\dot{x}(t)) + \O(h) \quad \text{as } h \rightarrow 0. \]

\item Consider a smooth function $g: T^{(2)}\R^N \rightarrow \R$, where $T^{(2)}\R^N$ is the second order tangent bundle of $\R^N$. A smooth function $\Phi: \big(\R^N\big)^3 \times \R_{>0} \rightarrow \R$ is a \emph{consistent discretization} of $g$ if for any smooth curve $x$ and for all $t$ there holds
\[ \Phi(x(t-h),x(t),x(t+h),h) = g(x(t),\dot{x}(t),\ddot{x}(t)) + \O(h)  \quad \text{as } h \rightarrow 0. \]
\end{enumerate}
\end{definition}

\begin{remark}
In Section \ref{sec-modlag} we will introduce the symbol $\simeq$ to denote asymptotic expansions. Here we prefer to use the usual equality sign because in practice we can often restrict to analytic curves. We reserve the symbol $\simeq$ for ``unavoidable'' asymptotic expansions, i.e.\@ situations where we generally do not have convergence even if all the relevant functions are analytic.
\end{remark}

There are two reasons why we put a stronger conditions in part $(a)$ than in part $(b)$. In Section \ref{sec-modlag} we will take $\Phi$ to be a discrete Lagrangian and we will need to write it as a power series to start our construction of a modified Lagrangian. Hence in the context of this paper, it is natural to include the existence of such an expansion in the notion of consistency. For the difference equations, for which part $(b)$ is the relevant definition, no such assumption is necessary. Additionally, the fact that the error term is given as a power series guarantees that its derivatives also $\O(h)$. We need this property in order to prove the following important observation.

\begin{proposition}\label{prop-consistency}
If $L_\disc: \left( \R^N  \right)^2 \times \R_{>0} \rightarrow \R$ is a consistent discretization of $\cL: T \R^N \rightarrow \R$, then the left hand side of the discrete Euler-Lagrange equation \eqref{disc-EL} is a consistent discretization of the left hand side of the continuous Euler-Lagrange equation \eqref{cont-EL}.
\end{proposition}
\begin{proof}
From the definition of consistency it follows that there exist functions $\widetilde{g}_i$ such that
\[
L_\disc(x(t),x(t+h),h) = \cL \!\left( x(t), \frac{x(t+h)-x(t)}{h} \right) + \sum_{i=1}^\infty h^i \widetilde{g}_i[x(t)]. \]
Taking a variation of the curve $x$ we find
\begin{align*}
&\D_1 L_\disc( x(t),x(t+h),h ) \delta x(t) + \D_2 L_\disc( x(t),x(t+h),h ) \delta x(t+h) \\
&= \der{\cL}{x} \!\left( x(t), \frac{x(t+h)-x(t)}{h} \right) \delta x(t) + \frac{1}{h}\der{\cL}{\dot{x}} \!\left( x(t), \frac{x(t+h)-x(t)}{h} \right) \delta x(t+h) \\
&\qquad - \frac{1}{h}\der{\cL}{\dot{x}} \!\left( x(t), \frac{x(t+h)-x(t)}{h} \right) \delta x(t) + \O(h),
\end{align*}
where $\der{\cL}{x}$ and $\der{\cL}{\dot{x}}$ denote the partial derivatives of $\cL$. Therefore,
\begin{align*}
\D_1 L_\disc( x(t),x(t+h),h ) &= \der{\cL}{x} \!\left( x(t), \frac{x(t+h)-x(t)}{h} \right) \\
&\qquad - \frac{1}{h}\der{\cL}{\dot{x}} \!\left( x(t), \frac{x(t+h)-x(t)}{h} \right)  + \O(h)
\end{align*}
and 
\[ \D_2 L_\disc( x(t),x(t+h),h ) = \frac{1}{h}\der{\cL}{\dot{x}} \!\left( x(t), \frac{x(t+h)-x(t)}{h} \right) + \O(h). \]
It follows that
\begin{align*}
\D_2 L_\disc&( x(t-h),x(t),h ) + \D_1 L_\disc( x(t),x(t+h),h ) \\
&= \der{\cL}{x} \!\left( x(t), \frac{x(t+h)-x(t)}{h} \right) - \frac{1}{h}\der{\cL}{\dot{x}} \!\left( x(t), \frac{x(t+h)-x(t)}{h} \right) \\
&\qquad + \frac{1}{h}\der{\cL}{\dot{x}} \!\left( x(t-h), \frac{x(t)-x(t-h)}{h} \right) + \O(h) \\
&= \der{\cL}{x}(x(t),\dot{x}(t)) - \frac{\d}{\d t} \der{\cL}{\dot{x}}(x(t),\dot{x}(t)) + \O(h). \qedhere
\end{align*}
\end{proof}

\begin{remark}
In most of the literature the discrete Lagrangian $L_\disc$ is chosen to be a consistent discretization of $h \cL$, rather than of $\cL$. 
\end{remark}

The discrete Lagrangian can be seen as a generating function for a symplectic map $(x_j,p_j) \mapsto (x_{j+1},p_{j+1})$, determined by 
\begin{equation}\label{disc-sympl} p_j = - \D_1 L_\disc(x_j,x_{j+1},h) \qquad \text{and} \qquad p_{j+1} = \D_2 L_\disc(x_j,x_{j+1},h).
\end{equation} 
In this way a variational integrator for $\cL$ leads to a symplectic integrator for the corresponding Hamiltonian system
\begin{equation}\label{hamiltonian} 
\dot{x} = \der{\cH}{p}, \qquad \dot{p} = -\der{\cH}{x},
\end{equation} 
where $p = \der{\cL}{\dot{x}}$ and the Hamilton function is given by $\cH = \inner{p}{\dot{x}} - \cL$, considered as a function of $x$ and $p$. The brackets $\inner{\cdot}{\cdot}$ denote the standard scalar product on $\R^N$.

\begin{example}\label{ex-methods}
There are many ways to obtain a discrete Lagrangian $L_\disc$ from a given continuous Lagrangian $\cL$. Some examples are:
\begin{enumerate}[$(a)$]
\item \label{Lh-MP} $\displaystyle L_\disc(x_j,x_{j+1},h) = \cL\!\left( \frac{x_j+x_{j+1}}{2} , \frac{x_{j+1}-x_j}{h} \right),$

in which case the symplectic map \eqref{disc-sympl} is the one obtained by applying the implicit midpoint rule to \eqref{hamiltonian}.

\item \label{Lh-SV} $\displaystyle 
L_\disc(x_j,x_{j+1},h) = \frac{1}{2} \cL\!\left( x_j, \frac{x_{j+1}-x_j}{h} \right) + \frac{1}{2} \cL\!\left( x_{j+1}, \frac{x_{j+1}-x_j}{h} \right), $

in which case the symplectic map \eqref{disc-sympl} is the one obtained by applying the St\"ormer-Verlet method to \eqref{hamiltonian}, assuming $\cL$ is separable. The St\"ormer-Verlet is a prime example of a geometric numerical integrator, as it can be used to illustrate many different concepts of geometric integration \cite{hairer2003geometric}.

\item \label{Lh-SE1} $\displaystyle L_\disc(x_j,x_{j+1},h) = \cL\!\left( x_j, \frac{x_{j+1}-x_j}{h} \right)$ 

or

\item \label{Lh-SE2} $\displaystyle
L_\disc(x_j,x_{j+1},h) = \cL\!\left( x_{j+1}, \frac{x_{j+1}-x_j}{h} \right),$

for which the symplectic maps \eqref{disc-sympl} are the ones obtained by applying the two variants of the symplectic Euler method to \eqref{hamiltonian}.
\end{enumerate}
\end{example}

\section{Modified equations}\label{sec-modeqn}

An important tool for studying the long-term behavior of numerical integrators is \emph{backward error analysis}. Instead of comparing a discrete solution $(x_j)_{j=0,\ldots,n}$ to a solution $x:[a,b] \rightarrow \R^N$ of the continuous system, backward error analysis compares the original differential equation to another differential equation satisfied by a curve $\widetilde{x}:[a,b] \rightarrow \R^N$ that interpolates the discrete solution. The latter differential equation is known as the \emph{modified equation}. 

\subsection{First order equations}

For first order equations the notion of modified equations is well-known, see for example \cite{calvo1994modified,hairer1994backward,moan2006modified,reich1999backward}, \cite[Chapter IX]{hairer2006geometric}, and the references therein. Nevertheless, defining a modified equation is a subtle matter. Let $\Psi(x_j,x_{j+1},h) = 0$ be a discretization of the differential equation. We would like to define a modified equation along the following lines.

\begin{pseudodef}
The differential equation $\dot{x} = f(x,h)$ is a \emph{modified equation} for the difference equation $\Psi(x_j,x_{j+1},h) = 0$ if for any solution $(x_j)_j$ of the difference equation, the differential equation has a solution $x$ that satisfies $x(jh)=x_j$ for all $j$.
\end{pseudodef}

However, we need to be more careful because the right hand side of the modified equation will generally be a power series in $h$ that does not converge. We write
\[ f(x,h) = f_0(x) + h f_1(x) +h^2 f_2(x) + \ldots, \]
and denote by $\cT_k$ the operator which truncates a power series in $h$ after order $k$,
\[ \cT_k \!\left( \sum_{i=0}^\infty f_i h^i \right) = \sum_{i=0}^k f_i h^i.\]
We call this the \emph{$k$-th truncation} of the power series. We say that two power series $f$ and $g$ are equal up to order $k$ if $\cT_k(f) = \cT_k(g)$, hence ``up to'' is to be understood as ``up to and including.'' 

Furthermore, we will need to consider families of curves parameterized by the step-size $h$, rather than just individual curves. Admissible families are those whose derivatives do not blow up as $h \rightarrow 0$.

\begin{definition}\label{def-admissible}
A family $(x_h)_{h \in \R_{>0}}$ of smooth curves $x_h: [a_h,b_h] \rightarrow \R^N$ is called \emph{admissible} if there exists a $h_{\mathrm{max}} > 0$ such that for each $k \geq 0$, $\big\| x_h^{(k)} \big\|_{\infty}$ is bounded as a function of $h \in (0,h_{\max}]$, where $\| \cdot \|_{\infty}$ denotes the supremum norm.
\end{definition}

Admissibility of a family of curves $(x_h)_h$ guarantees that in power series expansions like $x_h(t+h) = x_h(t) + h \dot{x}_h(t) + \frac{h^2}{2} \ddot{x}_h(t) + \ldots$ the asymptotic behavior of each term is determined by the exponent of $h$ in that term. This is essential in much of what follows and would not be the case for general families of curves. Now we are in a position to define a modified equation.

\begin{definition}\label{defi-firstorder}
Let $\Psi: \left( \R^N \right)^2 \times \R_{>0} \rightarrow \R^N$ be a consistent discretization of some $g: T\R^N \rightarrow \R$, with $\det \der{g}{\dot{x}} \neq 0$. The formal differential equation $\dot{x} = f(x,h)$, where
\[ f(x,h) = f_0(x) + h f_1(x) +h^2 f_2(x) + \ldots, \]
is a \emph{modified equation} for the difference equation $\Psi(x_j,x_{j+1},h) = 0$ if, for every $k$, every admissible family of solutions $(x_h)_h$ of the truncated differential equation 
\[ \dot{x}_h = \cT_k \left( f(x_h,h) \right), \qquad h \in \R_{>0} , \]
satisfies $\Psi(x_h(t),x_h(t+h),h) = \O(h^{k+1})$ as $h \rightarrow 0$ for all $t$.
\end{definition}

\begin{remark}
The discrete dynamics is invariant under scaling of the function $\Psi$ by a nonzero $h$-dependent factor, but the condition that $\Psi(x(t),x(t+h),h) = \O(h^{k+1})$ is not. This is not a problem because the scaling is constrained by the fact that $\Psi$ is a consistent discretization of some function $g$.
\end{remark}

\begin{proposition}
Let $\Psi: \left( \R^N \right)^2 \times \R_{>0} \rightarrow \R^N$ be a consistent discretization of some smooth $g: T\R^N \rightarrow \R^N$, with $\det \der{g}{\dot{x}} \neq 0$. Then the difference equation $\Psi(x_j,x_{j+1},h) = 0$ has a unique modified equation.
\end{proposition}

\begin{proof}
Because of the consistency, the Taylor expansion of $\Psi(x(t),x(t+h),h)$ takes the form
\begin{equation}\label{mod-taylor}
\Psi(x(t),x(t+h),h) = g(x,\dot{x}) + h g_1[x] + h^2 g_2[x] + \ldots ,
\end{equation}
where $g_1[x], g_2[x], \ldots$ depend on $x$ and its derivatives of arbitrary order. We look for a modified equation of the form 
\[ \dot{x} = f(x,h) = f_0(x) + hf_1(x) + h^2f_2(x) + \ldots.\]
This ansatz allows us to write the higher derivatives of $x$ as linear combinations of \emph{elementary differentials} \cite[Chapter III.1]{hairer2006geometric},
\[\begin{split}
\dot{x} &= f, \\
\ddot{x} &= f'f, \\
x^{(3)} &\stackrel[\vdots]{}{=} f''(f,f) + f'f'f, 
\end{split}\]
where a prime $'$ denotes differentiation with respect to $x$, and the arguments $x$ and $h$ of $f$ and its derivatives are omitted. Plugging these expressions into Equation \eqref{mod-taylor} we get
\[ \Psi(x(t),x(t+h),h) = g(x,f) + h g_1(x,f,f'f,f''(f,f) + f'f'f, \ldots) + \ldots, \]
where again the arguments of $f$ and its derivatives were omitted. By definition of modified equation this should be zero up to any order, 
\[ g(x,f) + h g_1(x,f,f'f,f''(f,f) + f'f'f, \ldots) + \ldots = 0. \]
The $h^k$-term of this expression is of the form
\[ \der{g}{\dot{x}} f_k + \text{terms depending only on } x, f_0, \ldots, f_{k-1},g,g_1,\ldots,g_k. \]
Since $g,g_1,g_2,\ldots$ are determined by $\Psi$, this gives us a recurrence relation for the $f_k$. 
\end{proof}

Some authors (e.g.\@ \citet*{calvo1994modified, hairer1994backward}) use the following property as their definition of a modified equation.

\begin{proposition}
Consider a difference equation of the form 
\[ x_{j+1} = x_j + h \Phi(x_j,x_{j+1}) \]
and let $(x_h)_h$ be an admissible family of solutions of the truncated modified equation $\dot{x}_h = \cT_k( f(x_h,h))$. Then 
\[ x_h(t+h) = x_h(t) + h \Phi(x_h(t),x_h(t+h)) + \O(h^{k+2}). \]
\end{proposition}

\begin{proof}
The difference equation can be written in the form $\Psi(x_j,x_{j+1},h) = 0$, where $\Psi(x_j,x_{j+1},h) = \frac{x_{j+1}-x_j}{h} - \Phi(x_j,x_{j+1})$ is a consistent discretization of $\dot{x} - \Phi(x,x)$. Hence any admissible family of solutions $(x_h)_h$ of the modified equation truncated after order $k$ satisfies
\begin{align*} 
x_h(t+h) - x_h(t) - h \Phi(x_h(t),x_h(t+h)) &= h \Psi(x_h(t),x_h(t+h),h) \\
&= \O(h^{k+2}). \qedhere
\end{align*}
\end{proof}

\subsection{Second order equations}

For the purposes of this paper we need to generalize Definition \ref{defi-firstorder}. Since we want to consider variational integrators, we need to introduce a notion of modified equations for second order difference equations.

\begin{definition}\label{defi-secondorder}
Let $\Psi: \big(\R^N\big)^3 \times \R_{>0} \rightarrow \R^N$ be a consistent discretization of $g: T^{(2)}\R^N \rightarrow \R^N$, with $\det \der{g}{\ddot{x}} \neq 0$. The formal differential equation $\ddot{x} = f(x,\dot{x},h)$, where
\[ f(x,\dot{x},h) = f_0(x,\dot{x}) + h f_1(x,\dot{x}) +h^2 f_2(x,\dot{x}) + \ldots, \]
is a \emph{modified equation} for the difference equation $\Psi(x_{j-1},x_j,x_{j+1},h) = 0$ if, for every $k$, every admissible family $(x_h)_h$ of solutions of the truncated differential equation 
\[ \ddot{x}_h = \cT_k \left( f(x_h,\dot{x}_h,h) \right) \]
satisfies $\Psi(x_h(t-h),x_h(t),x_h(t+h),h) = \O(h^{k+1})$ for all $t$.
\end{definition}
 
As in the first order case, we have existence and uniqueness. 
 
\begin{proposition}
Let $\Psi: \big(\R^N\big)^3 \times \R_{>0} \rightarrow \R^N$ be a consistent discretization of some smooth function $g: T^{(2)}\R^N \rightarrow \R^N$, with $\det \der{g}{\ddot{x}} \neq 0$. Then the difference equation $\Psi(x_{j-1},x_j,x_{j+1},h) = 0$ has a unique modified equation.
\end{proposition}
\begin{proof}
The Taylor expansion of $\Psi$ takes the form
\begin{equation}\label{mod-taylor-2}
\Psi(x(t-h),x(t),x(t+h),h) = g(x,\dot{x},\ddot{x}) + h g_1[x] + h^2 g_2[x] + \ldots ,
\end{equation}
where $g_1[x], g_2[x], \ldots$ depend on $x$ and its derivatives of arbitrary order. We look for a modified equation of the form 
\[ \begin{cases}
\dot{x} = v \\
\dot{v} = f(x,v,h) = f_0(x,v) + hf_1(x,v) + h^2f_2(x,v) + \ldots.
\end{cases} \]
This first order formulation of the modified equation allows us to write the higher derivatives of $x$ as linear combinations of elementary differentials \cite[Chapter III.2]{hairer2006geometric},
\begin{equation}\label{P-series-general}
\begin{split}
\ddot{x} &= f, \\
x^{(3)} &= f_x v + f_v f, \\
x^{(4)} &\stackrel[\vdots]{}{=} f_{xx} (v,v) + 2 f_{xv} (f,v) + f_x f + f_{vv} (f,f) + f_v f_x v + f_v f_v f,
\end{split}
\end{equation}
where the arguments $x$, $v$ and $h$ of $f$ and its derivatives were omitted, and the subscripts denote partial derivatives. Plugging these expressions into Equation \eqref{mod-taylor-2} we get 
\[ 0 = \Psi(x(t-h),x(t),x(t+h),h) = g(x,\dot{x},f) + h g_1(x,\dot{x},f,f_x \dot{x} + f_v f, \ldots) + \ldots \]
where again the arguments of $f$ and its derivatives were omitted. The $h^k$-term of this expression is of the form
\[ \der{g}{\ddot{x}} f_k + \text{terms depending only on } x, \dot{x}, f_0, \ldots, f_{k-1},g,g_1,\ldots,g_k.\]
Since $g,g_1,g_2,\ldots$ are determined by $\Psi_h$, this gives us a recurrence relation for the $f_k$. 
\end{proof}

\begin{example}\label{ex-modeqn}
Consider the differential equation $\ddot{x} = -U'(x)$, where $U: \R^N \rightarrow \R$ is some smooth potential, and its St\"ormer-Verlet discretization
\[ \frac{x_{j+1} - 2 x_j + x_{j-1}}{h^2} = -U'(x_j). \]
The modified equation is of the form
\[
\ddot{x} = f(x,h) = f_0(x,\dot{x}) + h^2 f_2 (x,\dot{x}) + \O(h^4).
\]
In general we should also include odd order terms, but in this example they all vanish because of the symmetry of the difference equation. We evaluate a smooth curve $x$ on a mesh of size $h$. In particular we consider $x_j = x(t)$ and
\[ x_{j \pm 1}=  x(t \pm h) = x \pm h \dot{x} + \frac{h^2}{2} \ddot{x} \pm \frac{h^3}{6} x^{(3)} + \frac{h^4}{24} x^{(4)} \pm \frac{h^5}{120} x^{(5)} + \O(h^6) . \]
We write $v = \dot{x}$, plug the above expansion into the difference equation, and replace derivatives using Equation \eqref{P-series-general}. This gives us
\begin{align*}
 -h^2 g(x)
&= h^2 \ddot{x} + \frac{h^4}{12} x^{(4)} + \O(h^6) \\
&= h^2 (f_0 + h^2 f_2) + \frac{h^4}{12} \big( f_{0,xx} (v,v) + 2 f_{0,xv} (f_0,v) + f_{0,x} f_0 \\
&\hspace{4cm} + f_{0,vv} (f_0,f_0) + f_{0,v} f_{0,x} v + f_{0,v} f_{0,v} f_0 \big) + \O(h^6),
\end{align*}
where the arguments $x$ and $v$ of the $f_i$ were omitted. The $h^2$-term of this equation gives us $f_0(x,v)=-U'(x)$. In particular, partial derivatives of $f_0$ with respect to $v$ are zero. The $h^4$-term then reduces to $f_2 = \frac{1}{12} (U^{(3)}(x) (v,v) - U''(x) U(x))$. We find the modified equation
\[
\ddot{x} = - U' + \frac{h^2}{12} \left( U^{(3)} (\dot{x},\dot{x}) - U'' U \right) + \O(h^4),
\]
where the argument $x$ of $U$ and of its derivatives has been omitted.

Observe that the truncation after the second order term of this modified equation is \emph{not} an Euler-Lagrange equation because the second order term $\frac{h^2}{12}\left( U^{(3)}(\dot{x},\dot{x}) - U''U' \right)$ contains first derivatives of $x$ but no second derivative of $x$. However, we will see that it can be obtained from an Euler-Lagrange equation by solving it for $\ddot{x}$ and truncating the resulting power series.
\end{example}

\begin{example}[Harmonic oscillator]
The simplest instance of the last example is the case that $x$ is real-valued and $U(x) = \frac{1}{2}x^2$, which gives us the difference equation
\[ \frac{x_{j+1} - 2 x_j + x_{j-1}}{h^2} = - x_j . \]
The modified equation for this difference equation is of the form
\begin{equation}\label{ansatz}
\ddot{x} = f(x,h) = f_0(x) + h^2 f_2 (x) + h^4 f_4 (x) + \O(h^6).
\end{equation}
The fact that the $f_i$ do not depend on $v = \dot{x}$ in this example vastly simplifies the calculations. It should be noted that this is very atypical behavior. In almost all other examples at least some $f_i$ do depend on $v = \dot{x}$. From Equation \eqref{ansatz} we obtain the following simplified form of the expressions in Equation \eqref{P-series-general}
\begin{align*}
x^{(3)} &= f'\dot{x}, \\
x^{(4)} &= f''\dot{x}^2 + f'f, \\
x^{(5)} &= f^{(3)}\dot{x}^3 + 3 f''f\dot{x} + (f')^2\dot{x}, \\
x^{(6)} &\stackrel[\vdots]{}{=} f^{(4)}\dot{x}^4 + 6 f^{(3)}f\dot{x}^2 + 5 f''f'\dot{x}^2 + 3 f''f^3 + (f')^2 f,
\end{align*}
where the arguments $x$ and $h$ of $f$ and its derivatives were omitted. If $x(t) = x_j$, then
\begin{align*}
x_{j \pm 1} =  x(t \pm h) &= x \pm h \dot{x} + \frac{h^2}{2} \ddot{x} \pm \frac{h^3}{6} x^{(3)} + \frac{h^4}{24} x^{(4)}  \\
&\qquad \pm \frac{h^5}{120} x^{(5)} + \frac{h^6}{720} x^{(6)} \pm \frac{h^7}{5040} x^{(7)} + \O(h^8) . 
\end{align*} 
Plugging this into the difference equation we find
\begin{align*}
 -h^2 x 
&= h^2 \ddot{x} + \frac{h^4}{12} x^{(4)} + \frac{h^6}{360} x^{(6)} + \O(h^8) \notag\\
&= h^2 \left( f_0 + h^2 f_2 + h^4 f_4 \right) \notag\\
&\quad + \frac{h^4}{12} \left( f_0''\dot{x}^2 + h^2f_2''\dot{x}^2 + f_0'f_0 + h^2 f_0'f_2 + h^2 f_2'f_0 \right) \notag\\
&\quad + \frac{h^6}{360} \left( f_0^{(4)}\dot{x}^4 + 6 f_0^{(3)}f_0\dot{x}^2 + 5 f_0''f_0'\dot{x}^2 + 3 f_0''f_0^2 + (f_0')^2f_0 \right) + \O(h^8).
\end{align*}
The $h^2$-term of this equation gives us $f_0(x)=-x$, and hence $f_0'(x) = -1$ and $f_0''(x) = 0$. The $h^4$-term then reduces to $f_2(x) = \frac{-x}{12}$, hence $f_2'(x) = -\frac{1}{12}$ and $f_2''(x) = 0$. Finally, the $h^6$-term gives
\[ f_4(x) = - \frac{1}{12}\left(\frac{x}{12}+\frac{x}{12}\right) + \frac{x}{360} = -\frac{x}{90}. \]
Therefore, the modified equation is
\[
\ddot{x} = - x - \frac{h^2}{12} x - \frac{h^4}{90} x + \O(h^6).
\]
In Figure \ref{fig-harmonic} we see that the solution of the fourth truncation of the modified equation agrees very well with the discrete flow, even with a large step-size.

\begin{figure}[t]
\begin{minipage}{.54\linewidth}
\includegraphics[width=\linewidth]{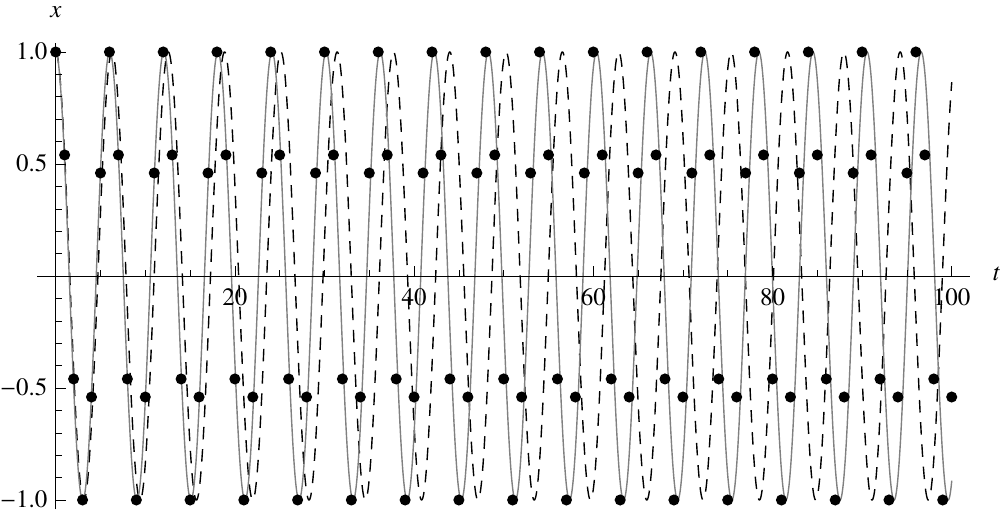}
\end{minipage}\hfill%
\begin{minipage}{.44\linewidth}
\includegraphics[width=\linewidth]{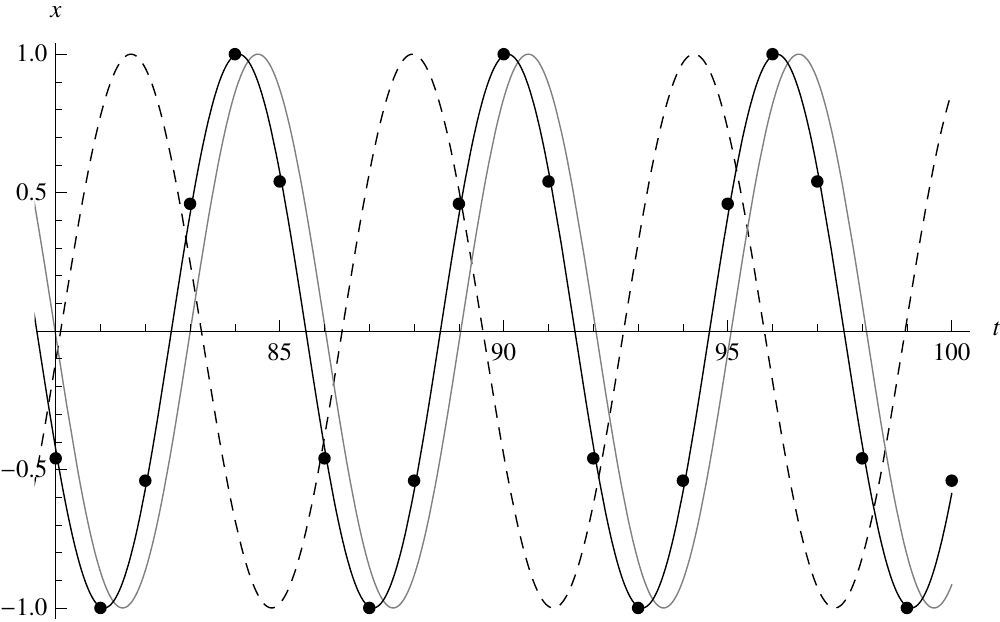}
\end{minipage}
\caption{The harmonic oscillator. Plotted are the exact solution (dashed line), the solution of the St\"ormer-Verlet discretization with step size $h=1$ (dots), and the solution of the modified equation truncated after order two (solid line). The right hand image is a magnification of the time interval $[80,100]$. Here the solution of the modified equation truncated after order four is drawn as well (dark solid line). The initial values are $x(0)=1$ and $\dot{x}(0)=0$ for the differential equations, and $x_0=1$ and $x_1=x(1)$ evaluated on the exact solution for the difference equation.}\label{fig-harmonic}
\end{figure}
\end{example}

\begin{remark}
In a surprising turn of events, solutions of the discrete system with step size $h=1$ are periodic, as can be observed in Figure \ref{fig-harmonic}. Apparently, solutions of the modified equation have a period of exactly $6$. This suggests that the modified equation for $h=1$ is $\ddot{x} = -\frac{\pi^2}{9} x$. Indeed, one can verify that in this example the modified Equation is given explicitly by
\[ \ddot{x} = - \left(\frac{2}{h} \arcsin\left(\frac{h}{2}\right) \right)^2 x. \]
This observation can be used as the basis of a simple proof of the well-known but very nontrivial expansion
\[ \left( \arcsin \frac{h}{2} \right)^2 = \frac{1}{2} \sum_{k=1}^\infty  \frac{(k-1)!^2}{(2k)!} h^{2k}. \]
The details of this argument are presented in \cite{vermeeren2015dynamical}.

For general $h \leq 1$ no periodicity is observed, but the solutions of the truncated modified equation are equally close or closer to the discrete solution.
\end{remark}

\section{Modified Lagrangians}\label{sec-modlag}

Given a variational integrator, we would like to find a Lagrangian that produces the modified equation as its Euler-Lagrange equation. The idea is to look for a modified Lagrangian $\cL_\mod ( x,\dot{x},h )$ such that the discrete Lagrangian is its \emph{exact discrete Lagrangian}, i.e.\@
\[ \int_0^h \cL_\mod ( x(t),\dot{x}(t),h ) \,\d t = h L_\disc (x_0, x_1,h), \]
where $x(t)$ is a critical curve for $\cL_\mod$ with $x(0) = x_0$ and $x(h) = x_1$. Since modified equations are generally non-convergent power series in $h$, the best we can hope for is to find such a modified Lagrangian up to an error of arbitrarily high order in $h$. Its Euler-Lagrange equation will then agree with the modified equation up to an error of the same order.

In intermediate steps of our construction we will find Lagrangians that depend on higher derivatives of the curve instead of just on $x$ and $\dot{x}$. Furthermore, the variational principle that these Lagrangians represent is unconventional: one looks for critical curves in a set of curves that need not be differentiable everywhere. Before starting the construction of a modified Lagrangian, we study this variational principle by itself. 

\subsection{Natural boundary conditions and meshed variational problems}

\begin{definition}
A \emph{classical variational problem} consists in finding critical curves of some action $\int_a^b \cL[x(t)] \,\d t$ in the set of smooth curves $\cC^\infty$.

A \emph{meshed variational problem} with mesh size $h$ consists in finding smooth critical curves of some action $\int_a^b \cL[x(t)] \,\d t$ in the set of piecewise smooth curves $\cC^{\cM,h}$ whose nonsmooth points occur at times that are an integer multiple of $h$ apart form each other,
\[
\cC^{\cM,h}
= \{ x \in \cC^0([a,b]) \mid \exists t_0 \in [a,b] : \forall t \in [a,b]: x \text{ not smooth at } t \ \Rightarrow\ t - t_0 \in h\mathbb{N} \}.
\]
This concept is illustrated in Figure \ref{fig-meshed}.
\end{definition}

\begin{figure}[b]
\begin{minipage}{.5\linewidth}
\centering
\begin{tikzpicture}[scale=.9, thick]
\clip (-.6,-.5) rectangle (5.5,2.5);
\draw[->] (-.2,-.3)--(5.2,-.3) node[right]{$t$};
\draw[->] (-.2,-.3)--(-.2,2.2) node[left]{$x$};
\draw (0,0) .. controls (1,1) and (4,1) .. (5,2);
\draw[dashed, darkgray]	(0,0) .. controls (1,-1) and (4,3) .. (5,2);
\draw[dashed]		 	(0,0) .. controls (0,2) and (4,1) .. (5,2);
\draw[dashed,gray] 		(0,0) .. controls (1,0) and (4,0) .. (5,2);
\node at (0,0) {$\bullet$};
\node at (5,2) {$\bullet$};
\end{tikzpicture}
\end{minipage}%
\begin{minipage}{.5\linewidth}
\centering
\begin{tikzpicture}[scale=.9, thick]
\clip (-.6,-.5) rectangle (5.5,2.5);
\draw[->] (-.2,-.3)--(5.2,-.3) node[right]{$t$};
\draw[->] (-.2,-.3)--(-.2,2.2) node[left]{$x$};

\draw (0,0) .. controls (1,1) and (4,1) .. (5,2);
\draw[dashed] 			(0,0) .. controls (0,2) and (4,1) .. (5,2);
\draw[dashed,darkgray]	(1.5,.7) .. controls (1.8,0) and (2.3,0) .. (2.5,.5);
\draw[dashed,darkgray]	(2.5,.5) .. controls (2.8,0) and (3.3,0) .. (3.5,1.25);
\draw[dashed,gray]		(4,1.4) .. controls (4,0) and (5,1) .. (5,2);
\node at (0,0) {$\bullet$};
\node at (1.5,.7) {$\bullet$};
\node at (2.5,.5) {$\bullet$};
\node at (3.5,1.25) {$\bullet$};
\node at (4,1.4) {$\bullet$};
\node at (5,2) {$\bullet$};
\end{tikzpicture}
\end{minipage}%
\caption{A smooth curve and a few of its variations for classical variational problem (left) and a meshed variational problem (right).}
\label{fig-meshed}
\end{figure}
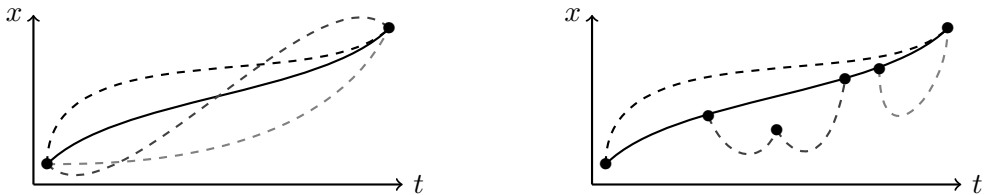

Consider a classical variational problem on the interval $[a,b]$ with a Lagrange function $\cL[x]$. The condition for criticality reads
\begin{equation}\label{ibp}
\int_a^b \frac{\delta \cL}{\delta x} \delta x \,\d t + \sum_{j=0}^{\infty} \frac{\delta \cL}{\delta x^{(j+1)}} \delta x^{(j)} \bigg|_a^b = 0 ,
\end{equation}
where 
\[ \frac{\delta \cL}{\delta x^{(j)}} = \sum_{i=0}^\infty (-1)^i \frac{\d^i}{\d t^i} \der{\cL}{x^{(j+i)}} \]
are variational derivatives of $\cL$. 

We assume that each of the quantities $x(a),x(b),\dot{x}(a),\dot{x}(b),\ddot{x}(a),\ddot{x}(b), \ldots$ is either fixed independently of the others or left completely free. Depending on which of those are fixed and which are left free, the following necessary and sufficient conditions follow from \eqref{ibp}:
\begin{flalign*}
\begin{split}
(a) \ & \frac{\delta \cL}{\delta x} = 0. \\
(b) \ & \text{If } x^{(j)}(a) \text{ is free, } \left. \frac{\delta \cL}{\delta x^{(j+1)}} \right|_{t=a} = 0. \qquad \text{If } x^{(j)}(b) \text{ is free, } \left. \frac{\delta \cL}{\delta x^{(j+1)}}  \right|_{t=b} = 0. 
\end{split}
\end{flalign*}
Condition $(a)$ is the Euler-Lagrange equation. Conditions $(b)$ are known as the \emph{natural boundary conditions}.

Now consider a meshed variational problem on the interval $[a,b]$ with Lagrange function $\cL[x]$. A necessary condition for criticality is that on each interval $[t, t + h] \subset [a,b]$ the corresponding classical variational problem, with boundary conditions on $x$ but not on the derivatives, is solved. This gives the conditions that on the whole time interval $[a,b]$:
\begin{flalign}\label{meshed-el}
\begin{split}
(a) \ & \frac{\delta \cL}{\delta x} = 0. \\
(b) \ & \forall j \geq 2: \frac{\delta \cL}{\delta x^{(j)}} = 0, \\ &\text{or equivalently: } \forall j \geq 2: \frac{\partial \cL}{\partial x^{(j)}} = \frac{\delta \cL}{\delta x^{(j)}} - \frac{\d}{\d t}\frac{\delta \cL}{\delta x^{(j+1)}} = 0 .
\end{split}&
\end{flalign}
These conditions are also sufficient, because any variation consistent with the meshed structure can be written as the sum of a smooth variation on $[a,b]$ and variations on intervals $[t, t + h]$ that vanish at the endpoints. In analogy with the classical case we call \eqref{meshed-el}$(b)$ the \emph{natural interior conditions}. They can also be seen as a version of the Weierstrass-Erdmann corner conditions, where the time of a corner is not allowed to be varied, but every point is a corner.

Since the Euler-Lagrange equation \eqref{meshed-el}$(a)$ together with suitable boundary conditions already determine a unique solution, meshed variational problems are overdetermined. This should not be surprising. After all we are looking for critical curves in a set $\cC^{\cM,h}$ of piecewise smooth curves, but at the same time require the critical curve to be in the subset  $\cC^\infty \subset \cC^{\cM,h}$ of smooth curves.

\subsection{A meshed modified Lagrangian}

Now we begin the construction of a modified Lagrangian from a given discrete Lagrangian $L_\disc$ that is a consistent discretization of some continuous Lagrangian. Using a Taylor expansion we can write the discrete Lagrangian $L_\disc\!\left(x\!\left(t-\frac{h}{2}\right),x\!\left(t+\frac{h}{2}\right),h\right)$ as a function of a smooth curve $x$ and its derivatives, all evaluated at time $t$,
\begin{equation} \label{sca-disclag} 
\begin{split}
\cL_\disc([x(t)],h) := L_\disc \!\Bigg( & x(t) - \frac{h}{2}\dot{x}(t) + \frac{1}{2} \left(\frac{h}{2}\right)^2 \ddot{x}(t) - \ldots , \\
&\quad x(t) + \frac{h}{2}\dot{x}(t) + \frac{1}{2} \left(\frac{h}{2}\right)^2 \ddot{x}(t) + \ldots ,\ h \Bigg).
\end{split}
\end{equation}
From Equation \eqref{sca-disclag} we proceed by expanding $L_\disc(\cdot,\cdot,h)$ around the point $(x(t),x(t))$ to write $\cL_\disc([x],h)$ explicitly as a power series in $h$. 

\begin{remark}
We could also have chosen $t-\frac{h}{2}$, $t+\frac{h}{2}$, or any other point in the interval $\left[ t-\frac{h}{2}, t+\frac{h}{2} \right]$ to expand around. Choosing the midpoint has the computational advantage that the expansions of some common terms like $\frac{1}{h}\left( x\!\left(t+\frac{h}{2}\right) - x\!\left(t-\frac{h}{2}\right) \right)$ and $\frac{1}{2} \left( x\!\left(t-\frac{h}{2}\right) + x\!\left(t+\frac{h}{2}\right) \right)$ only contain even powers of $h$. 
\end{remark}

\begin{proposition}\label{prop-Ldisc-orders}
If the discrete Lagrangian $L_\disc$ is a consistent discretization of some $\cL(x,\dot{x})$, then the $h^k$-term of $\cL_\disc$ depends on $x, \dot{x}, \ldots, x^{(k+1)}$, but not on higher derivatives of $x$. 
\end{proposition}
\begin{proof}
Let $y = x + \frac{h^2}{8} \ddot{x} + \ldots$ and $z = \dot{x} + \frac{h^2}{24} x^{(3)} + \ldots$. Then the power series expansion of $\cL_\disc([x],h)$ takes the form
\begin{align*}
\cL_\disc([x],h) 
&= L_\disc \!\left( y-\frac{h}{2}z, y+\frac{h}{2}z ,h \right)  \\
&= \cL(y,z) + h g_1(y,z) + h^2 g_2(y,z) + \ldots. \qedhere
\end{align*}
\end{proof}

We want to write the discrete action
\[ S_\disc = \sum_{j=1}^n h L_\disc(x(jh-h),x(jh),h) = \sum_{j=1}^n h \cL_\disc\!\left(\left[x \!\left( jh - \tfrac{h}{2}\right)\right], h \right) \] 
as an integral. To do this we require a lemma.

\begin{lemma}\label{lemma-EM}
For any smooth function $f: \mathbb{R} \rightarrow \mathbb{R}^N$ we have
\begin{align*}
\sum_{j=1}^n h f \!\left( jh - \frac{h}{2} \right) 
&\simeq \int_0^{nh} \left( \sum_{i=0}^\infty h^{2i} \left(2^{1-2i}-1\right) \frac{B_{2i}}{(2i)!} f^{(2i)}(t) \right) \d t ,
\end{align*} 
where $B_i$ are the Bernoulli numbers. The symbol $\simeq$ denotes an asymptotic expansion for $h \rightarrow 0$. In general, the power series in the right hand side does not converge.
\end{lemma}
\begin{remark}
The first few terms can easily be obtained by Taylor expansion. We have
\begin{align*}
\int_0^h f(t) \,\d t &= \int_0^h f\!\left(\tfrac{h}{2}\right) + \left(t-\tfrac{h}{2}\right) f'\!\left(\tfrac{h}{2}\right) + \frac{1}{2}\left(t-\tfrac{h}{2}\right)^2 f''\!\left(\tfrac{h}{2}\right) + \O(t^3) \,\d t \\
&= h f\!\left(\tfrac{h}{2}\right) + \frac{h^3}{24} f''\!\left(\tfrac{h}{2}\right)+ \O(h^4) \\
&= h f\!\left(\tfrac{h}{2}\right) + \int_0^h \frac{h^2}{24} f''\!\left(\tfrac{h}{2}\right) \d t + \O(h^4)  \\
&= h f\!\left(\tfrac{h}{2}\right) + \int_0^h  \frac{h^2}{24} f''(t) \,\d t + \O(h^4) ,
\end{align*}
which gives the result up to order 2 after summation:
\[ \sum_{j=1}^n h f \!\left( jh - \frac{h}{2} \right) 
= \int_0^{nh} f(t) - \frac{h^2}{24} f''(t) \,\d t + \O(nh^4). \]
We could prove the general statement by an iteration of this procedure, but here we give a shorter albeit slightly less elementary proof.
\end{remark}
\begin{proof}[Proof of Lemma \ref{lemma-EM}]
The Euler-Maclaurin formula \cite[Section 23.1]{abramowitz1972handbook} gives the following asymptotic expansion: 
\[ \sum_{j=1}^{n-1} g(j) 
\simeq \int_0^{n} g(t) \,\d t - \frac{1}{2}(g(0)+g(n)) + \sum_{i=1}^\infty \frac{B_{2i}}{(2i)!} \!\left( g^{(2i-1)}(n)-g^{(2i-1)}(0) \right) \]
for any smooth function $g: \mathbb{R} \rightarrow \mathbb{R}^N$. If we double $n$ in this formula, we get
\[ \sum_{j=1}^{2n-1} g(j) 
\simeq \int_0^{2n} g(t) \,\d t - \frac{1}{2}(g(0)+g(2n)) + \sum_{i=1}^\infty \frac{B_{2i}}{(2i)!} \!\left( g^{(2i-1)}(2n)-g^{(2i-1)}(0) \right) \!. \]
If we double the the argument of $g$ instead, we get
\begin{align*}
\sum_{j=1}^{n-1} g(2j) 
\simeq \int_0^n g(2t) \,\d t &- \frac{1}{2}(g(0)+g(2n)) + \sum_{i=1}^\infty 2^{2i-1} \frac{B_{2i}}{(2i)!} \left( g^{(2i-1)}(2n)-g^{(2i-1)}(0) \right) \!. 
\end{align*}
Taking the difference yields
\[ \sum_{j=1}^{n} g(2j-1) 
\simeq \int_0^n g(2t) \,\d t + \sum_{i=1}^\infty \left(1-2^{2i-1}\right) \frac{B_{2i}}{(2i)!} \left( g^{(2i-1)}(2n)-g^{(2i-1)}(0) \right) \!, \]
hence 
\[ \sum_{j=1}^{n} g(2j-1) \simeq \int_0^n \left( g(2t) + \sum_{i=1}^\infty \left(2-2^{2i}\right) \frac{B_{2i}}{(2i)!} g^{(2i)}(2t) \right)\! \d t. \]
Now set $f(t) = g\left(\frac{2}{h}t \right)$. Then
\[\sum_{j=1}^{n} f \left(hj-\frac{h}{2}\right) 
\simeq \int_0^n \left( f(ht) + \sum_{i=1}^\infty \left(2-2^{2i}\right) \frac{B_{2i}}{(2i)!} \left(\frac{h}{2}\right)^{2i} f^{(2i)}(ht) \right)\! \d t,\]
which is equivalent to the claimed result. 
\end{proof}

\begin{definition}
We call the formal power series
\begin{align*}
\cL_\mesh([x(t)],h)
&:= \sum_{i=0}^\infty \left(2^{1-2i}-1\right) \frac{h^{2i} B_{2i}}{(2i)!} \frac{\d^{2i}}{\d t^{2i}} \cL_\disc([x(t)],h) \\
&= \cL_\disc([x(t)],h) - \frac{h^2}{24} \frac{\d^2}{\d t^2} \cL_\disc([x(t)],h) + \frac{7h^4}{5760} \frac{\d^4}{\d t^4} \cL_\disc([x(t)],h) + \ldots
\end{align*}
the \emph{meshed modified Lagrangian} of $L_\disc$. 
\end{definition}

Note that the higher order terms of the meshed modified Lagrangian do not contribute to the Euler-Lagrange equations because they are time derivatives. However, they do contribute to the natural interior conditions. Furthermore they are needed to have (formal) equality between the discrete and the meshed modified action,
\[ S_\disc((x(jh))_j,h) = \sum_j h \cL_\disc\!\left( \left[x \left( jh - \tfrac{h}{2}\right)\right] ,h \right) \simeq \int_0^{nh} \cL_\mesh([x(t)],h) \,\d t \]
for any smooth curve $x$. This implies that if $x$ is a curve such that $(x(t_0+jh))_j$ is critical for the discrete action, then $x$ formally solves the meshed variational problem for $\cL_\mesh$.

The rest of this section is devoted to constructing a classical, first-order Lagrangian $\cL_\mod: T \R^N \rightarrow \R$ for the modified equation. From this point on our construction differs significantly from the one presented in \cite{oliver2012new}. First we have to do some analysis.

\subsection{Properties of admissible families of curves}

Recall from Definition \ref{def-admissible} that a family of curves is called admissible if their derivatives of any order are bounded as $h \rightarrow 0$. An admissible family of real valued curves (i.e.\@ with $N = 1$) is called an admissible family of functions. In particular, for a family of Lagrangians $(\cL_h)_h$ that is given by a power series in $h$ and an admissible family of curves $(x_h)_h$, the compositions $\cL_h[x_h]$ form an admissible family of functions.

\begin{lemma}
If $(x_h)_h$ is an admissible family of curves, then for every $k \in \mathbb{N}$ the family of derivatives $\big( x_h^{(k)} \big)_h$ is admissible as well.
\end{lemma}
\begin{proof}
This follows immediately from the definition of admissibility. 
\end{proof}

\begin{lemma}\label{lemma-admissible1}
Let $(f_h)_{h \in \R_{>0}}$ be an admissible family of functions on the same domain $[a,b]$ and let $(h_k)_k$ be a sequence with $h_k \rightarrow 0$. If $\lim_{k \rightarrow \infty} \| f_{h_k} \|_{\infty} = 0$, then $\lim_{k \rightarrow \infty} \| f_{h_k}' \|_{\infty} = 0$. (And hence $\lim_{k \rightarrow \infty} \big\| f_{h_k}^{(n)} \big\|_{\infty} = 0$ for all $n$.)
\end{lemma}

\begin{proof}
Suppose towards a contradiction that there exists an $\varepsilon > 0$ and a sequence $(t_k)_k$ such that $|f_{h_k}'(t_k)| > \varepsilon$. Without loss of generality we can assume $\varepsilon < 1$. Since  $\lim_{k \rightarrow \infty} \| f_{h_k} \|_{\infty} = 0$, for every $j \in \mathbb{N}$ we can find a $k \in \mathbb{N}$ such that for all $\ell \geq k$ there holds $\|f_{h_\ell}\|_{\infty} < \frac{1}{8} \varepsilon^{j+1}$.

We claim that for every $\ell \geq k$ there exists an $s_\ell \in \left[ t_\ell - \frac{1}{2}\varepsilon^j, t_\ell \right]$ such that \linebreak$\big| f_{h_\ell}'(t_\ell) - f_{h_\ell}'(s_\ell) \big| \geq \frac{\varepsilon}{2}$. Indeed, if this were not the case there would hold that
\begin{align*}
\left| f_{h_\ell}(t_\ell) - f_{h_\ell}\left(t_\ell - \tfrac{1}{2}\varepsilon^j \right) \right| 
&\geq \frac{\varepsilon^j}{2} \inf_{t \in [ t_\ell - \frac{1}{2}\varepsilon^j, t_\ell ]} \left| f_{h_\ell}'(t) \right| \\
&> \frac{\varepsilon^j}{2}\left(f_{h_\ell}'(t_\ell) - \frac{\varepsilon}{2}\right) > \frac{\varepsilon^{j+1}}{4},
\end{align*}
which contradicts the fact that $\|f_{h_\ell}\|_{\infty} < \frac{1}{8} \varepsilon^{j+1}$.

Since such an $s_\ell$ exists, we can find an $r_\ell \in [s_\ell, t_\ell] \subset \left[ t_\ell - \frac{1}{2}\varepsilon^j, t_\ell \right]$ such that $| f_{h_\ell}''(r_\ell) | > \frac{\varepsilon}{2} \big/ \frac{\varepsilon^j}{2} = \varepsilon^{1-j}$. It follows that $\limsup_{\ell \rightarrow \infty} \| f_{h_\ell}'' \|_{\infty} \geq \lim_{j \rightarrow \infty} \varepsilon^{1-j} = \infty$, which contradicts the assumption that $(f_h)_h$ is admissible. 
\end{proof}

\begin{lemma}\label{lemma-admissible2}
\begin{enumerate}[$(a)$]
\item Let $(f_h)_{h \in \R_{>0}}$ be an admissible family of functions on the same domain $[a,b]$. If $\| f_h \|_\infty = \O(h^\ell)$, then for all $k$ there holds $\big\| f_h^{(k)} \big\|_\infty = \O(h^\ell)$.

\item Let $(f_h)_{h \in \R_{>0}}$ be an admissible family of functions on a shrinking domain $[a_h,b_h] = [a_h,a_h+h]$. If $\| f_h \|_\infty = \O(h^\ell)$, then for all $k$ there holds $\big\| f_h^{(k)} \big\|_\infty = \O(h^{\ell-k})$.
\end{enumerate}

\end{lemma}
\begin{proof}
\begin{enumerate}[$(a)$]
\item Since the derivatives of an admissible family of functions form an admissible family, it is sufficient to show this for $k=1$.

Assume towards a contradiction that $\| f_h' \|_\infty$ is not $\O(h^\ell)$. Then there exists a sequence $(h_j)_j$ with $h_j \rightarrow 0$ such that $\|f_{h_j}'\|_{\infty} > j h^\ell$. Hence 
\[ \lim_{j \rightarrow \infty} \left\| \frac{f_{h_j}}{\|f_{h_j}'\|_{\infty}} \right\|_{\infty}= 0 . \] 
But then a contradiction follows from Lemma \ref{lemma-admissible1}, applied to the family $\left(f_h / \|f'_h\|_\infty \right)_h$:
\[ 1 = \lim_{j \rightarrow \infty} \left\| \frac{f'_{h_j}}{\|f_{h_j}'\|_{\infty}} \right\|_{\infty}  = 0 . \] 

\item Consider the functions $g_h: [0,1] \rightarrow \R^N$ defined by rescaling $f_h$: 
\[ g_h(t) = f_h(a_h+ht). \]
Then $\big\|g_h^{(k)}\big\|_\infty = h^k \big\|f_h^{(k)}\big\|_\infty$, so the $g_h$ form an admissible family. Hence from part $(a)$ it follows that $h^k \big\|f_h^{(k)}\big\|_\infty = \O(h^\ell)$. 
\end{enumerate}
\end{proof}

\begin{lemma}\label{lemma-bootstrap}
Let $(f_h)_{h \in \R_{>0}}$ be an admissible family of functions with the same domain $[a,b]$. 
\begin{enumerate}[$(a)$]
\item If $\sup_t \left| f_h(t) + f_h(t+h) \right| = \O(h^\ell)$, then $\| f_h \|_\infty = \O(h^\ell)$.

\item If $\sup_t \left| \int_t^{t+h} f_h(s) \,\d s \right| = \O(h^{\ell+1})$, then $\| f_h \|_\infty = \O(h^\ell)$.
\end{enumerate}

\end{lemma}
\begin{proof}
\begin{enumerate}[$(a)$]

\item We proceed by induction on $\ell$. If $\ell=0$ the claim follows from the definition of admissibility. Assume the statement holds for $\ell-1$. Observe that $\sup_t | f_h(t+h) - f_h(t) | = \O(h)$, so $\sup_t | f_h(t) + f_h(t+h) | = \O(h^\ell)$ implies $\| f_h \|_\infty = \O(h)$, which by Lemma \ref{lemma-admissible2} implies that $\big\| f_h^{(k)}(t) \big\|_\infty = \O(h)$ for every $k$. Therefore $\left( \frac{1}{h}f_h(t) \right)_h$ is an admissible family of functions. Since $\sup_t \left| \frac{1}{h}f_h(t) + \frac{1}{h}f_h(t+h) \right| = \O(h^{\ell-1})$, the induction hypothesis implies that $\left\| \frac{1}{h}f_h \right\|_\infty = \O(h^{\ell-1})$.

\item Again we use induction on $\ell$. And again the claim follows from the definition of admissibility if $\ell=0$. Assume it holds for $\ell-1$ and let $g_h$ be the antiderivative of $f_h$ with $g_h(a) = 0$. Then
\begin{equation}\label{bootstrap-mesh}
\begin{split}
\max_{k} |g_h(a+kh)| &= \max_{k} |g_h(a+kh) - g_h(a)| \\
&= \max_{k} \Bigg| \sum_{j=1}^k \int_{a+(j-1)h}^{a+jh} f_h(s) \,\d s \,\Bigg| = \O(h^\ell),
\end{split}
\end{equation}
where that maximum is taken over all integers $k$ such that $a+kh \in [a,b]$, hence $k = \O(h^{-1})$. By the induction hypothesis we have $\| f_h \|_\infty = \O(h^{\ell-1})$, so
\begin{equation}\label{bootstrap-close}
\sup_{s,t \,:\, |t-s|<h} |g_h(s) - g_h(t)| = \O(h^{\ell}).
\end{equation}
Together, Equations \eqref{bootstrap-mesh} and \eqref{bootstrap-close} imply that $\| g_h \|_\infty = \O(h^{\ell}) $. Since the $g_h$ form an admissible family it follows that $\| f_h \|_\infty = \| g_h' \|_\infty = \O(h^{\ell})$.
\qedhere
\end{enumerate}
\end{proof}

Studying meshed variational problems for admissible families of curves instead of individual piecewise smooth curves is much more subtle. The reason for this is that the higher derivatives of variations on a mesh interval $[t,t+h]$ tend to increase without bound as $h \rightarrow 0$. Such variations take us outside the set of admissible families and are therefore not allowed. The next subsection provides us with a framework to circumvent this.

\subsection{$k$-critical families of curves}

Modified equations generally are nonconvergent power series and so are modified Lagrangians. To make sense of these analytically we need to truncate the power series. It will be useful to allow an unspecified truncation error in the notion of a critical curve.

\begin{definition}
\begin{enumerate}[$(a)$]
\item An admissible family of curves $x_h:[a,b]\rightarrow \R$ is \emph{$k$-critical} for some family of actions $S_h = \int_a^b \cL_h \,\d t$ if for any family of smooth variations $\delta x_h$ there holds $\delta S_h = \O\big( h^{k+1} \, \|\delta x_h\|_1 \big)$. The set of $k$-critical families of curves is denoted by $\cC_k(\cL_h)$.

\item An admissible family of curves $x_h:[a,b]\rightarrow \R$ is \emph{meshed $k$-critical} for some family of actions $S_h = \int_a^b \cL_h \,\d t$ if for any family of piecewise smooth variations $\delta x_h$ of $x_h$, with nonsmooth points in a mesh with size $h$, there holds $\delta S_h = \O\big( h^{k+1} \, \|\delta x_h\|_1 \big)$. The set of meshed $k$-critical families of curves is denoted by $\cC_k^\cM(\cL_h)$.

\item A family of discrete curves $(x_j)_j$ ($h$ omitted to ease notation) is \emph{$k$-critical} for some family of actions $S_\disc = \sum_j L_\disc(x_j,x_{j+1},h)$ if for any family of variations of $(x_j)_j$ there holds $\delta S_\disc = \O\big( h^{k+1} \, \| (\delta x_j)_j \| \big)$, where $\|(\delta x_j)_j\| = \sum_j h |\delta x_j|$.
\end{enumerate}

\smallskip
\noindent In each of definitions above we assume that a full set of boundary conditions is provided and that the variations respect these boundary conditions. 
\end{definition}

\begin{remark}
The scaling of the norm in the discrete case is such that for any smooth variation $\delta x$ there holds $\|\delta x\|_1 = \big(1 + \O(h)\big) \|(\delta x(jh))_j\|$. 
\end{remark}

We can characterize $k$-critical families of curves by a natural relaxation of the usual criticality conditions.

\begin{lemma}\label{lemma-k-crit}
\begin{enumerate}[$(a)$]
\item An admissible family of curves $x_h:[a,b]\rightarrow \R$ is $k$-critical for the family of actions $S_h = \int_a^b \cL_h \,\d t$ if and only if it satisfies the corresponding Euler-Lagrange equations with a defect of order $\O(h^{k+1})$:
\begin{equation}\label{k-crit}
\left\| \frac{\delta \cL_h}{\delta x} \right\|_\infty = \O(h^{k+1}) .
\end{equation}

\item An admissible family of curves $x_h:[a,b]\rightarrow \R$ is meshed $k$-critical for the family of actions $S_h = \int_a^b \cL_h \,\d t$ if and only if it satisfies 
\begin{equation}\label{k-crit-mesh}
\left\| \frac{\delta \cL_h}{\delta x} \right\|_\infty = \O(h^{k+1}) \quad \text{and} \quad \left\|  \der{\cL_h}{x^{(\ell)}} \right\|_\infty = \O(h^{k+\ell+1}) \quad \text{for all } \ell \geq 2. 
\end{equation}

\item A family of discrete curves $(x_j)_j$ is $k$-critical for the family of actions $S_\disc = \sum_j L_\disc(x_j,x_{j+1},h)$ if and only if it satisfies the discrete Euler-Lagrange equations with a defect of order $\O(h^{k+1})$:
\[
\D_2 L_\disc(x_{j-1},x_j,h) + \D_1 L_\disc(x_j,x_{j+1},h) = \O(h^{k+1}). 
\]
\end{enumerate}
\end{lemma}

\begin{proof}
\begin{enumerate}[$(a)$]
\item Consider a family of Lagrangians $\cL_h$ and a smooth curve $x$. It is sufficient to consider variations that have a fixed 1-norm, say $\| \delta x_h \|_1 = 1$. For any such family of variations $(\delta x_h)_h$ we have
\[
\delta S_h = \int_a^b \frac{\delta \cL_h}{\delta x} \delta x_h \,\d t . \]
It follows that Equation \eqref{k-crit} holds if and only if $\delta S_h = \O(h^{k+1}) = \O\big( h^{k+1} \|\delta x_h\|_1 \big)$ for all families of variations with $\| \delta x_h \|_1 = 1$.

\item Any variation can be written as the sum of variations supported on single mesh intervals and a smooth variation. It is sufficient to look at these types of variations separately. Smooth variations are treated as in $(a)$. For a variation supported on a mesh interval $[t_0, t_0+h]$ we find
\[ \delta S_h = \int_{t_0}^{t_0 + h} \frac{\delta \cL_h}{\delta x} \delta x_h \,\d t + \sum_{i=2}^\infty \sum_{j=0}^{i-2} (-1)^j \frac{\d^j}{\d t^j} \der{\cL_h}{x^{(i)}} \delta x_h^{(i-j-1)} \bigg|_{t_0}^{t_0 + h}. \]
Note that we did not include $j=i-1$ in the summation range, because the variation $\delta x_h$ must vanish at the endpoints $t_0$ and $t_0+h$.
If $\| \delta x_h \|_\infty = \O(h^\ell)$ for some $\ell$ , then by Lemma \ref{lemma-admissible2}$(b)$ we have $\big\| \delta x_h^{(i-j-1)} \big\|_\infty = \O(h^{\ell-i+j+1})$. Then the conditions \eqref{k-crit-mesh} imply that $\delta S_h = \O(h^{k+\ell+2})$. Since $\| \delta x_h \|_1 = \O(h^{\ell+1})$, this is sufficient for meshed $k$-criticality.

By considering smooth variations as in part $(a)$, we can conclude that also in this case the Euler-Lagrange equations up to order $k$ are necessary conditions. More subtle to show is the necessity of the natural interior conditions. The difficulty is that the derivatives of variations supported on a mesh interval $[t,t+h]$ are usually unbounded as $h \rightarrow 0$, so the set of admissible families of such variations is rather small. 

We will use induction on $m$ to show that
\begin{align*}
\forall m \geq 0:\ \forall k \geq 0:\ \forall \ell \geq \max\{2,m-k-1\} &:\ \left\| \der{\cL_h}{x^{(\ell)}} \right\|_\infty = \O(h^{m}) \\ 
&\quad \text{ on $k$-critical families}.
\end{align*}
For $m=0$ this follows from the admissibility of the family of curves.

Now fix some $M$ and suppose the claim holds for $m < M$. Take any $k \geq 0$ and $\ell \geq \max\{2,M-k-1\}$. Note that this implies $k + \ell \geq M-1$. To construct admissible variations we consider the family of polynomials $p_{\ell,h}(t)$ of degree $\ell$ in $t$ that satisfies
\begin{align*}
& p_{\ell,h}(0) = p_{\ell,h}(h) = 0, \\
& p'_{\ell,h}(0) = p'_{\ell,h}(h), \quad p''_{\ell,h}(0) = p''_{\ell,h}(h), \quad \ldots \quad p_{\ell,h}^{(\ell-2)}(0) = p_{\ell,h}^{(\ell-2)}(h), \\
& p_{\ell,h}^{(\ell)} \equiv 1.
\end{align*}
For each $\ell,h$ these conditions uniquely define a polynomial, because they are equivalent to $\ell+1$ independent linear equations in the coefficients of $p_{\ell,h}$. Note that these polynomials satisfy the scaling relation $p_{\ell,h}(h t) = h^\ell p_{\ell,1}(t)$, from which it follows that $\max_{[0,h]} \big| p_{\ell,h}^{(j)} \big| = \O(h^{\ell-j})$. In particular, we have that $p_{\ell,h}^{(\ell-1)}(t) = t - \frac{h}{2}$. 

Pick any time $t_0$. Consider the family of variations 
\[ \delta x_h(t) = p_{\ell,h}(t -t_0) \, \mathds{1}_{[t_0 ,t_0 + h ]}(t) \, v, \]
where $\mathds{1}_A$ denotes the indicator function of $A$ and $v \in \R^N$ is a constant vector. Then $\|\delta x_h\|_1 = \O(h^{\ell+1})$, hence for meshed $k$-critical families of curves there holds that
\begin{align*}
\delta S &= \int_{t_0}^{t_0 + h} \sum_{i=1}^\ell (-1)^i \frac{\d^i}{\d t^i} \der{\cL_h}{x^{(i)}} \delta x_h \,\d t \\
&\qquad + \sum_{i=2}^\ell \sum_{j=0}^{i-2} (-1)^j \frac{\d^j}{\d t^j} \der{\cL_h}{x^{(i)}} \delta x_h^{(i-j-1)} \bigg|_{t_0}^{t_0 + h} 
= \O(h^{k+\ell+2}).
\end{align*}
In fact there even holds that,
\begin{equation}\label{meshed-var}
\begin{split}
\sup_{t_0} \Bigg| & \int_{t_0}^{t_0 + h} \sum_{i=1}^\ell (-1)^i \frac{\d^i}{\d t^i} \der{\cL_h}{x^{(i)}} \delta x_h \,\d t \\
&+ \sum_{i=2}^\ell \sum_{j=0}^{i-2} (-1)^j \frac{\d^j}{\d t^j} \der{\cL_h}{x^{(i)}} \delta x_h^{(i-j-1)} \bigg|_{t_0}^{t_0 + h} \Bigg|
 = \O(h^{k+\ell+2}),
\end{split}
\end{equation}
which can be proved using shifted variations for which $t_0$ depends on $h$. We have already established that $\left\| \frac{\delta \cL_h}{\delta x} \right\|_\infty = \O(h^{k+1})$ on meshed $k$-critical families and the induction hypothesis implies that $\left\| \der{\cL_h}{x^{(i)}} \right\|_\infty = \O(h^{M-1})$ for $i > \ell$ on meshed $k$-critical families. It follows that
\[ \sup_{t_0} \left| \int_{t_0}^{t_0 + h} \sum_{i=1}^\ell (-1)^i \frac{\d^i}{\d t^i} \der{\cL_h}{x^{(i)}} \delta x_h \,\d t \right| = \O(h^{M+1})\]
and hence (since $k+\ell+2 \geq M+1$)
\begin{equation}\label{bdy-terms}
\sup_{t_0} \left| \sum_{i=2}^\ell \sum_{j=0}^{i-2} (-1)^j \frac{\d^j}{\d t^j} \der{\cL_h}{x^{(i)}} \delta x_h^{(i-j-1)} \bigg|_{t_0}^{t_0 + h} \right| = \O(h^{M+1}) .
\end{equation}
Using the fact that $\big\| \delta x_h^{(i-j-1)} \big\|_\infty = \O(h^{\ell-i+j+1})$ and the induction hypothesis we find that for all $i$ and $j$ in the range of this sum, except $(i,j) = (\ell,0)$,
\[ \left\| \frac{\d^j}{\d t^j} \der{\cL_h}{x^{(i)}} \delta x_h^{(i-j-1)} \right\|_\infty = \O(h^{M+1}), \]
hence
\[ \sup_{t_0} \left| \frac{\d^j}{\d t^j} \der{\cL_h}{x^{(i)}} \delta x_h^{(i-j-1)}  \bigg|_{t_0}^{t_0 + h} \right| = \O(h^{M+1}). \]
Equation \eqref{bdy-terms} now implies that the term with $(i,j)=(\ell,0)$ satisfies the same order condition:
\[ \sup_{t_0} \left| \frac{h}{2}\left( \der{\cL_h}{x^{(\ell)}} \bigg|_{t_0} + \der{\cL_h}{x^{(\ell)}} \bigg|_{t_0 + h} \right) \right|
= \O(h^{M+1}) . \]
By Lemma \ref{lemma-bootstrap}$(a)$ it follows that
\[ \left\| \der{\cL_h}{x^{(\ell)}} \right\|_\infty = \O(h^{M}). \]

This concludes the induction step and thus the proof that the interior conditions, up to the appropriate order, are necessary for $k$-criticality. 

\item If the family of discrete curves $(x_j)_j$ is $k$-critical, then
\begin{align*}
\sum_j h \left( \D_2 L_\disc(x_{j-1},x_j) + \D_1 L_\disc(x_j,x_{j+1}) \right) \delta x_j &= \delta S_\disc \\
&= \O(h^{k+1} \|(\delta x_j)_j\|).
\end{align*}
For some index $\ell$, set $\delta x_\ell = \frac{1}{h}$ and $\delta x_j = 0$ for $j \neq \ell$, then $\| (\delta x_j)_j \| = 1$. It follows that 
\[ \D_2 L_\disc(x_{\ell-1},x_\ell) + \D_1 L_\disc(x_\ell,x_{\ell+1}) = \O(h^{k+1}). \]

On the other hand, for any family of variations $(\delta x_j)_j$ with $\| (\delta x_j)_j \| = 1$ we have
\begin{align*}
|\delta S_\disc| &\leq \sum_j h \left|\left( \D_2 L_\disc(x_{j-1},x_j) + \big. \D_1 L_\disc(x_j,x_{j+1}) \right) \delta x_j \right| \\
&\leq \bigg( \sum_j h |\delta x_j| \bigg) \max_j \left( \left| \D_2 L_\disc(x_{j-1},x_j) + \big. \D_1 L_\disc(x_j,x_{j+1}) \right| \right) \\
&= \max_j \left( \left| \D_2 L_\disc(x_{j-1},x_j) + \big. \D_1 L_\disc(x_j,x_{j+1}) \right| \right).
\end{align*}
Hence $(x_j)_j$ is $k$-critical if the discrete Euler-Lagrange equations are satisfied up to order $k$.
\qedhere
\end{enumerate}
\end{proof}

\subsection{Properties of the meshed modified Lagrangian}

Now that we have established the analytic framework, it is time to list some important properties of the meshed modified Lagrangian. 

\begin{lemma}
Let $L_\disc$ be a consistent discretization of a regular Lagrangian $\cL(x,\dot{x})$. Then the zeroth order term of the modified Lagrangian is the original continuous Lagrangian, i.e.\@ $\cL_\mesh([x],h) = \cL (x,\dot{x}) +\O(h)$. 
\end{lemma}
\begin{proof}
We have 
\begin{align*}
\cL_\mesh([x(t)],h)
&= \cL_\disc([x(t)],h) + \O(h^2) \\
&= L_\disc \left( x \left(t - \tfrac{h}{2} \right) , x \left(t + \tfrac{h}{2} \right) ,h \right) + \O(h^2) \\
&= \cL(x(t),\dot{x}(t)) + \O(h). \qedhere
\end{align*} 
\end{proof}

An essential property of the meshed modified Lagrangian is that any curve that solves the Euler-Lagrange equations automatically satisfies the natural interior conditions.

\begin{lemma}\label{lemma-partials}
If a family of curves $(x_h)_h$ satisfies the Euler-Lagrange equations of $\cL_\mesh$ up to order $k$ then it satisfies the natural interior conditions
\[ \left\| \der{\cL_\mesh}{x^{(\ell)}} \right\|_\infty = \O(h^{k+\ell+1}) \qquad \text{for all } \ell \geq 2. \]
In other words, every $k$-critical family of curves for $\cL_\mesh$ is also meshed $k$-critical, $\cC_k^\cM(\cL_\mesh) = \cC_k(\cL_\mesh)$. 
\end{lemma}
\begin{proof}
Consider the same family of polynomials $p_{\ell,h}(t)$ as in the proof of Lemma \ref{lemma-k-crit}$(b)$ and the corresponding family of variations $\delta x_h(t) = p_{\ell,h}(t -t_0) \mathds{1}_{[t_0 ,t_0 + h ]} v$. Since these variations do not affect the discrete action 
\[ S_\disc = \sum_j L_\disc(x_h(t_0 + (j-1)h),x_h(t_0 + jh) , h), \]
there holds for every curve that
\[ \int_a^b \sum_{i=0}^\ell \der{\cL_\mesh}{x^{(i)}} \delta x_h^{(i)} \,\d t \simeq \delta S_\disc = 0. \]
In particular this implies Equation \eqref{meshed-var}. Since the Euler-Lagrange equations are satisfied up to order $k$, we can proceed exactly as in the proof of Lemma \ref{lemma-k-crit}$(b)$. 
\end{proof}

The modified Lagrangian depends on fewer derivatives of $x$ than $\cL_\disc$ (cf.\@ Proposition \ref{prop-Ldisc-orders}):

\begin{proposition}\label{prop-orders}
For $\ell \geq 1$ the $h^\ell$-term of $\cL_\mesh$ (as a power series in $h$) depends on $x, \dot{x}, \ldots, x^{(\ell)}$, but not on higher derivatives of $x$. 
\end{proposition}

\begin{proof}
Any admissible family of curves satisfies the Euler-Lagrange equations up to order $-1$:
\[ \left\| \frac{\delta \cL_\mesh}{\delta x} \right\|_\infty = \O(1). \]
Hence it follows from Lemma \ref{lemma-partials} that for all $\ell \geq 2$:
\[ \left\| \der{\cL_\mesh}{x^{(\ell)}} \right\|_\infty = \O(h^\ell), \]
which implies that $x^{(\ell)}$ can only occur in $\cL_\mesh$ in terms of order at least $h^\ell$. 
\end{proof}

\subsection{The modified equation}

From Lemma \ref{lemma-partials} it follows that $k$-critical families of curves for $\cL_\mesh$ satisfy the Euler-Lagrange equation
\[
\der{\cL_\mesh}{x} - \frac{\d}{\d t} \der{\cL_\mesh}{\dot{x}} = \O(h^{k+1})
\]
even though $\cL_\mesh$ depends on higher derivatives of $x$. By Proposition \ref{prop-orders}, this equation takes the form
\begin{equation}\label{high-order-EL}
\begin{split}
&\cE_0(x,\dot{x},\ddot{x}) + h \cE_1(x,\dot{x},\ddot{x}) \\
&\quad + h^2 \cE_2(x,\dot{x},\ddot{x},x^{(3)}) + \ldots + h^k \cE_k(x,\dot{x},\ldots,x^{(k+1)}) = \O(h^{k+1}).
\end{split}
\end{equation}
If we replace the error term by an exact zero, this is a singularly perturbed equation, whose solutions in general have increasingly steep boundary layers as $h \rightarrow 0$. However, the condition that $(x_h)_h$ is an admissible family of curves excludes this behavior and allows us to write Equation \eqref{high-order-EL} as a second order differential equation with an $\O(h^{k+1})$ defect. This is done by a simple recursion.

If $L_\disc$ is a consistent discretization of some regular continuous Lagrangian, then for sufficiently small $h$ we can solve $\cE_0(x,\dot{x},\ddot{x}) + h \cE_1(x,\dot{x},\ddot{x}) = \O(h^2)$ for $\ddot{x}$, say 
\[ \ddot{x} = F_1(x,\dot{x},h) + \O(h^2) . \]
Now suppose that Equation \eqref{high-order-EL} implies $\ddot{x} = F_k(x,\dot{x},h) + \O(h^{k+1})$. Then there exist functions $F_k^2,F_k^3,\ldots: T\R \times \R_{>0} \rightarrow \R$ such that
\begin{align*}
&\ddot{x} = F_k^2(x,\dot{x},h) + \O(h^{k+1}) = F_k(x,\dot{x},h) + \O(h^{k+1}), \\
&x^{(3)} = F_k^3(x,\dot{x},h) + \O(h^{k+1}), \\
&x^{(4)} = F_k^4(x,\dot{x},h) + \O(h^{k+1}), \ \ldots .
\end{align*}
Then Equation \eqref{high-order-EL} (with $k$ replaced by $k+2$) implies
\begin{align*}
&\cE_0(x,\dot{x},\ddot{x}) + h \cE_1(x,\dot{x},\ddot{x}) \\
&\quad + \left( h^2 \cE_2(x,\dot{x},\ddot{x},x^{(3)}) + \ldots + h^{k+2} \cE_k(x,\dot{x},\ldots,x^{(k+3)}) \right) 
\!\Big|_{x^{(j)} = F_k^j(x,\dot{x},h)} \\
&\hspace{9.5cm} = \O(h^{k+3}).
\end{align*}
After making the replacements, the terms between the parentheses only depend on $x$ and its first derivative. Hence we can solve this equation for $\ddot{x}$ to find an expression of the form $\ddot{x} = F_{k+2}(x,\dot{x},h) + \O(h^{k+3})$.

Note that each step of this recursion increases the order of accuracy by two. This is the case because we only replace derivatives in terms of second and higher order.

\subsection{A classical modified Lagrangian}

\begin{definition}
The \emph{modified Lagrangian} is the formal power series
\[ \cL_\mod(x,\dot{x},h) = \cL_\mesh([x],h)
\Big|_{\ddot{x} = f(x,\dot{x},h), \ x^{(3)} = \frac{\d}{\d t} f(x,\dot{x},h), \ \ldots} , \]
where $\ddot{x} = f(x,\dot{x},h)$ is the modified equation. The $k$-th truncation of the modified Lagrangian is denoted by $\cL_{\mod,k}$,
\[ \cL_{\mod,k}(x,\dot{x},h) = \cT_k\!\left( \cL_\mod(x,\dot{x},h) \right) = \cT_k\left( \cL_\mesh([x],h)\Big|_{x^{(j)} = F_{k-2}^j(x,\dot{x},h)} \right) , \]
where $\cT_k $ denotes truncation after the $h^k$-term. 
\end{definition}

From the definition it follows that $\cL_{\mod,k} ( x_h,\dot{x}_h,h ) = \cL_\mesh([x_h],h) + \O(h^{k+1})$ for families of curves $(x_h)_h$ that are $k$-critical for $\cL_\mesh$. Since this does not hold for general curves, it does not immediately imply that the Euler-Lagrange equations of both Lagrangians agree up to order $k$. Nevertheless, this property holds true.

\begin{lemma}\label{lemma-1stord}
The meshed modified Lagrangian $\cL_\mesh([x],h)$ and the modified Lagrangian $\cL_{\mod,k}(x,\dot{x},h)$ have the same $k$-critical families of curves.
\end{lemma}
\begin{proof}
We proceed by induction. Since $\cL_\mesh([x],h) = \cL_{\mod,0}(x,\dot{x},h) + \O(h)$ they have the same $0$-critical families of curves, $C_0(\cL_\mesh) = C_0(\cL_{\mod,0})$. Now suppose that $C_{k-1}(\cL_\mesh) = C_{k-1}(\cL_{\mod,k-1})$. Since $k$-critical families of curves are also $(k-1)$-critical, this set contains all $k$-critical families of curves of both $\cL_\mesh$ and $\cL_{\mod,k}$.

For every family of curves in $C_{k-1}(\cL_\mesh)$ we have by Lemma \ref{lemma-partials}
\begin{align*}
\der{\cL_{\mod,k}}{x} &= \left.\left( \der{\cL_\mesh}{x} + \sum_{\ell=2}^\infty \der{\cL_\mesh}{x^{(\ell)}} \der{F_k^\ell(x,\dot{x},h)}{x} + \O(h^{k + 1}) \right)\right|_{x^{(j)} = F_{k-1}^j(x,\dot{x},h)} \\
&= \der{\cL_\mesh}{x} + \O(h^{k + 1})
\end{align*}
and
\begin{align*}
\der{\cL_{\mod,k}}{\dot x} &= \left.\left( \der{\cL_\mesh}{\dot x} + \sum_{\ell=2}^\infty \der{\cL_\mesh}{x^{(\ell)}}\der{F_k^\ell(x,\dot{x},h)}{\dot x} + \O(h^{k + 1}) \right)\right|_{x^{(j)} = F_{k-1}^j(x,\dot{x},h)}\\
&= \der{\cL_\mesh}{\dot x} + \O(h^{k + 1}),
\end{align*}
hence
\[ \der{\cL_{\mod,k}}{x} - \frac{\d}{\d t} \der{\cL_{\mod,k}}{\dot x} =  \sum_{\ell=0}^\infty (-1)^\ell \frac{\d^\ell}{\d t^\ell} \der{\cL_\mesh}{x^{(\ell)}} + \O(h^{k+1}), \]
which shows that $C_k(\cL_\mesh) = C_k(\cL_{\mod,k})$. 
\end{proof}

We now arrive at our main result: up to truncations, the modified equation is Lagrangian in the classical sense.

\begin{theorem}\label{thm-1stord}
For a discrete Lagrangian $L_\disc$ that is a consistent discretization of a regular Lagrangian $\cL$, the $k$-th truncation of the Euler-Lagrange equation of $\cL_{\mod,k}(x,\dot{x},h)$, solved for $\ddot{x}$, is the $k$-th truncation of the modified equation.
\end{theorem}
\begin{proof}
Let $(x_h)_h$ be an admissible family of solutions of the $k$-th truncation of the Euler-Lagrange equation for $\cL_{\mod,k}$. Then $(x_h)_h$ is $k$-critical for the family of actions $\int_a^b \cL_{\mod,k}(x,\dot{x},h) \,\d t$. Consider the discrete curve $(x_j)_j := (x(jh))_j$, an admissible family of variations $\delta x_h$ of $x_h$ and the corresponding family of variations $(\delta x_j)_j$ of $(x_j)_j$ with $\delta x_j = \delta x(jh)$. Then $\|\delta x\|_1 = \big(1 + \O(h)\big) \|(\delta x(jh))_j\|$.

By Lemma \ref{lemma-1stord}, the family $(x_h)_h$ is $k$-critical for $\cL_\mesh([x(t)],h)$. By construction, the actions $\sum_j h L_\disc(y(jh),y((j+1)h),h)$ and $\int_a^b \cL_\mesh([y(t)],h) \,\d t$ are (formally) equal for any smooth curve $y$. Therefore
\begin{align*} 
\delta S_\disc &= \delta \sum_j h L_\disc(x(jh),x((j+1)h),h) \\
&\simeq \delta \int_a^b \cL_\mesh([x(t)],h) \d t
= \O\big( h^{k+1} \|\delta x\|_1 \big) = \O\big( h^{k+1} \|(\delta x_j)_j\| \big),
\end{align*}
so the family of discrete curves $(x(jh))_j$ is $k$-critical for the family of discrete actions $S_\disc(\cdot,h)$. Hence $(x(jh))_j$ satisfies the discrete Euler-Lagrange equation up to order $h^k$, i.e.
\[ \D_2 L_\disc(x(t-h),x(t),h) + \D_1 L_\disc(x(t),x(t+h),h) = \O(h^{k+1}) . \]
By Proposition \ref{prop-consistency} the left hand side of this expression is a consistent discretization of the continuous Euler-Lagrange equation, so this order condition is the one defining a modified equation as in Definition \ref{defi-secondorder}. 
\end{proof}

Theorem \ref{thm-1stord} provides an alternative proof of the following well-known result \cite[Chapter IX.3]{hairer2006geometric}.

\begin{corollary}
If a symplectic method is applied to a Hamiltonian system with a regular Hamiltonian, then any truncation of the resulting modified equation is Hamiltonian.
\end{corollary}
\begin{proof}
By applying the (inverse) Legendre transformation we obtain a Lagrangian system with regular Lagrangian. We can apply Theorem \ref{thm-1stord} to this Lagrangian to find a modified Lagrangian $\cL_{\mod,k} (x,\dot{x},h)$. For sufficiently small $h$, the modified Lagrangian is regular as well. Therefore we can take the Legendre transformation again and obtain a modified Hamiltonian. 
\end{proof}

\section{Examples}\label{sec-ex}

\subsection{St\"ormer-Verlet discretization of mechanical Lagrangians}

A Lagrangian $\cL:T\R^N \rightarrow \R$ is called \emph{separable} if there exists functions $K$ and $U$ such that $\cL(x,\dot{x}) = K(\dot{x}) - U(x)$. The Euler-Lagrange equation of such a Lagrangian is
\[ \der{^2 K(\dot{x})}{\dot{x}^2} \; \ddot{x} = - \der{U(x)}{x}. \]
If $\cL$ is separable, then the discrete Lagrangians $(\ref{Lh-SV})$, $(\ref{Lh-SE1})$, and $(\ref{Lh-SE2})$ from Example \ref{ex-methods} are equivalent (but their discrete Legendre transforms are different). A separable Lagrangian with $K(\dot{x}) = \frac{1}{2} |\dot{x}|^2$ is called a \emph{mechanical Lagrangian}.

\subsubsection{Second order}\label{sssect-sv}

We consider some mechanical Lagrangian $\cL(x,\dot{x}) = \frac{1}{2} |\dot{x}|^2 - U(x)$ and use the St\"ormer-Verlet discretization, whose discrete Lagrangian is given by Example \ref{ex-methods}$(\ref{Lh-SV})$,
\[ L_\disc(x_j,x_{j+1},h) = \frac{1}{2} \left| \frac{x_{j+1}-x_j}{h} \right|^2 - \frac{1}{2} U\left( x_{j} \right) - \frac{1}{2} U\left( x_{j+1} \right).\]
Its Euler-Lagrange equation is
\[ \frac{x_{j+1} - 2 x_j + x_{j-1}}{h^2} = - U'(x_j). \]
We have
\begin{align*}
\cL_\disc([x],h)
&= \left| \dot{x} + \frac{h^2}{24} x^{(3)} + \ldots \right|^2
- \frac{1}{2} U \bigg(x - \frac{h}{2}\dot{x} + \frac{1}{2} \left(\frac{h}{2}\right)^2 \ddot{x} - \ldots \bigg) \\
&\hspace{3.6cm}- \frac{1}{2} U \bigg(x + \frac{h}{2}\dot{x} + \frac{1}{2} \left(\frac{h}{2}\right)^2 \ddot{x} + \ldots \bigg) \\
&= \frac{1}{2} |\dot{x}|^2 - U + \frac{h^2}{24} \left( \binner{\dot{x}}{x^{(3)}} - 3 U'\ddot{x} -3 U''(\dot{x},\dot{x}) \right) + \O(h^4), 
\end{align*}
where the argument $x$ of $U$ has been omitted.

From $\cL_\disc([x],h)$ we calculate the meshed modified Lagrangian as follows:
\begin{align}
\cL_\mesh([x],h) 
&= \cL_\disc([x],h) - \frac{h^2}{24} \frac{\d^2}{\d t^2} \cL_\disc([x],h) + \O(h^4) \notag \\
&= \frac{1}{2} |\dot{x}|^2 - U + \frac{h^2}{24} \left( \binner{\dot{x}}{x^{(3)}} - 3 U' \ddot{x} - 3 U''(\dot{x},\dot{x}) \right)  \notag\\
&\hspace{2.2cm} - \frac{h^2}{24} \left( |\ddot{x}|^2 + \binner{\dot{x}}{x^{(3)}} - U' \ddot{x} - U''(\dot{x},\dot{x}) \right) + \O(h^4) \notag\\
&= \frac{1}{2} |\dot{x}|^2 - U + \frac{h^2}{24} \left( -|\ddot{x}|^2 - 2 U' \ddot{x} - 2 U''(\dot{x},\dot{x}) \right) + \O(h^4). \label{SV-meshlag}
\end{align}

The modified equation up to second order is then obtained from
\[
\O(h^{4}) = \der{\cL_\mesh}{x} - \frac{\d}{\d t} \der{\cL_\mesh}{\dot{x}} = - \ddot{x} - U' + \frac{h^2}{12} \left(U''\ddot{x} + U^{(3)}(\dot{x},\dot{x}) \right).
\]
We solve this recursively for $\ddot{x}$. In the leading order we have $\ddot{x} = -U'$, the original equation, so in the second order term we can substitute $\ddot{x} = -U'$. Hence the modified equation is 
\begin{equation}\label{mech-modeqn}
\ddot{x} = - U' + \frac{h^2}{12} \left(U^{(3)}(\dot{x},\dot{x}) - U''U' \right) + \O(h^4),
\end{equation}
as we already found in Example \ref{ex-modeqn}.

To obtain the classical modified Lagrangian we need to replace higher derivatives in the meshed modified Lagrangian \eqref{SV-meshlag} using the modified Equation \eqref{mech-modeqn}. In fact, to get the modified Lagrangian up two order two, we only need the leading order term $\ddot{x} = - U'$ of the modified equation. We find
\begin{equation} \label{SV-modlag}
\cL_{\mod,3}(x,\dot{x},h) = \frac{1}{2} |\dot{x}|^2 - U + \frac{h^2}{24} \left( \left|U'\right|^2 - 2 U''(\dot{x},\dot{x}) \right).
\end{equation}

Observe that the first order modified Lagrangian $\cL_{\mod,3}(x,\dot{x},h)$ is not separable for general $U$ because the term $U''(\dot{x},\dot{x})$ depends on both $x$ and $\dot{x}$. The Euler-Lagrange equation of $\cL_{\mod,3}$ is
\[
 -\ddot{x} -U' +\frac{h^2}{12}\left( U''U' + U^{(3)}(\dot{x},\dot{x}) + 2  U''\ddot{x} \right) = 0 .
\]
Note that this equation does not contain an error term. However when we solve it for $\ddot{x}$ we again get \eqref{mech-modeqn}, including the $\O(h^4)$ error term. In other words, $\ddot{x} = -U' +\frac{h^2}{12}\left( U^{(3)}(\dot{x},\dot{x}) - U''U' \right)$ is \emph{not} the Euler-Lagrange equation for $\cL_{\mod,3}$, but it is $\O(h^4)$-close to it.

\begin{example}[Kepler Problem]
The Lagrangian of a point-mass in a (classical) gravitational potential is
\[ \cL(x,\dot{x}) = \frac{1}{2} |\dot{x}|^2 + \frac{1}{|x|}. \]
Its Euler-Lagrange equation is $\ddot{x} = - \frac{x}{|x|^3}$. The St\"ormer-Verlet discretization of this system is
\[ x_{j+1} = 2 x_{j} - x_{j-1} - h^2 \frac{x_j}{|x_j|^3}. \]
Plugging the potential $U(x) = - \frac{1}{|x|}$ into Equation \eqref{SV-modlag}, we find that the third truncation of the modified Lagrangian is
\[ \cL_{\mod,3}(x,\dot{x},h) = \frac{1}{2} |\dot{x}|^2 + \frac{1}{|x|} + \frac{h^2}{24}\left( \frac{1}{|x|^4} - 2 \frac{|\dot{x}|^2}{|x|^3} + 6 \frac{\inner{x}{\dot{x}}^2}{|x|^5}\right). \]
The modified equation reads
\[ \ddot{x} = -\frac{x}{|x|^3} + \frac{h^2}{6}\frac{x}{|x|^6} - \frac{h^2}{2} \frac{\inner{x}{\dot{x}} \dot{x}}{|x|^5} - \frac{h^2}{4} \frac{|\dot{x}|^2 x}{|x|^5} + \frac{5h^2}{4} \frac{\inner{x}{\dot{x}}^2 x}{|x|^7} + \O(h^4). \]
Figure \ref{fig-kepler} shows that this modified equation exhibits the correct qualitative long-time behavior: the orbit precesses counterclockwise (whereas the mass orbits clockwise). The precession is marginally slower for the modified equation. In \cite{vermeeren2016numerical} modified Lagrangians are used to estimate the numerical precession rates of different variational integrators.

\begin{figure}[t]
\begin{minipage}{.49\linewidth}
\includegraphics[width=\linewidth]{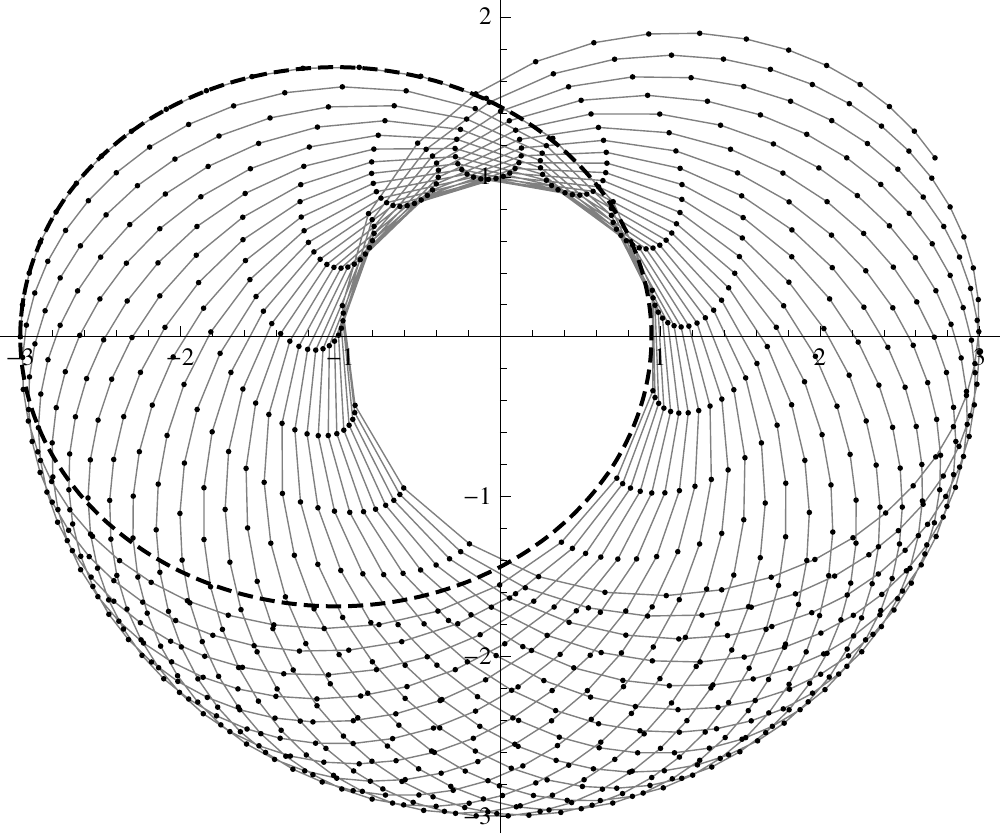}
\end{minipage}\hfill%
\begin{minipage}{.49\linewidth}
\includegraphics[width=\linewidth]{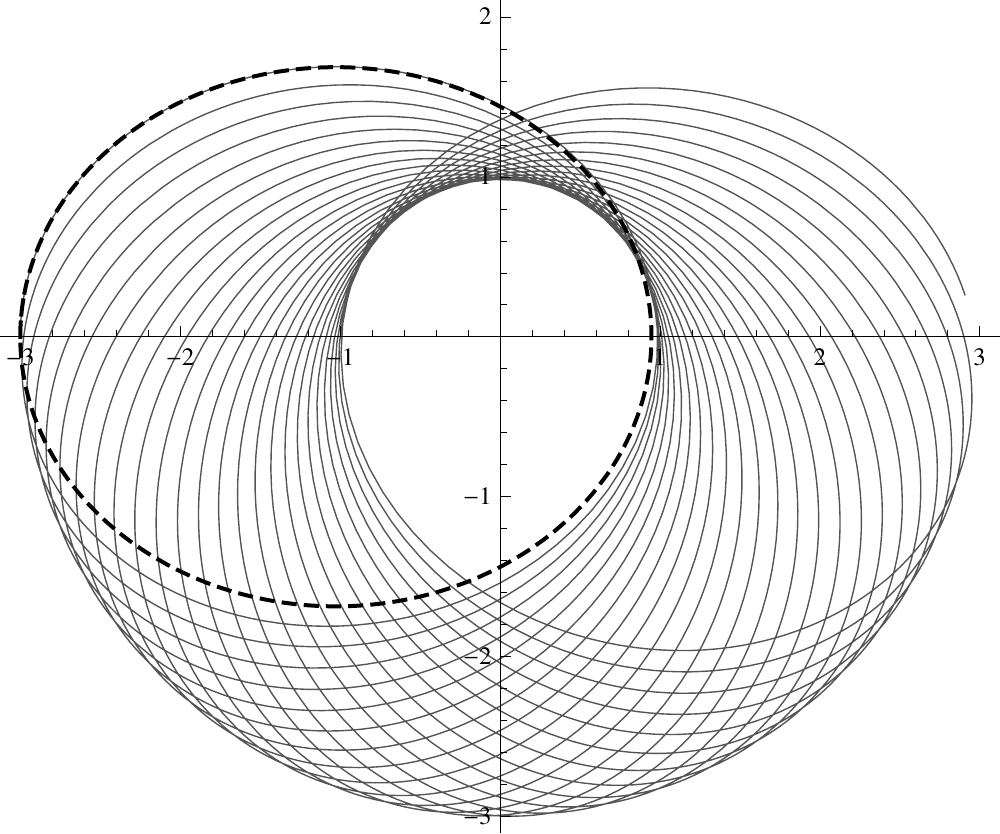}
\end{minipage}
\caption{The St\"ormer-Verlet discretization of the Kepler problem with step size $h=0.5$. In both images the dashed ellipse is the exact solution of the original equation. On the left the discrete solution is shown as well. Shown on the right is a solution of the modified equation, truncated after order 2, over the same time interval. The initial position is $(-3,0)$. The solutions of the original and modified equation have initial velocity $(0,0.4)$. For the discrete system the second point is chosen to be the evaluation at time $h$ of the exact solution of the original equation.}\label{fig-kepler}
\end{figure}
\end{example}

\subsubsection{Fourth order}
We extend the calculations of Section \ref{sssect-sv} to include the $h^4$-terms. We find
\begin{align*}
\cL_\disc([x],h)
&= \frac{1}{2} |\dot{x}|^2 - U + \frac{h^2}{24} \big(-3 U''(\dot{x},\dot{x}) - 3 U'\ddot{x} + \binner{\dot{x}}{x^{(3)}} \big) \\
&\quad + \frac{h^4}{5760} \big( -45 U''(\ddot{x},\ddot{x}) - 90 U^{(3)}(\ddot{x},\dot{x},\dot{x}) - 60 U''(x^{(3)},\dot{x})  \\
&\hspace{2cm} + 5 \big|x^{(3)}\big|^2 - 15 U^{(4)}(\dot{x},\dot{x},\dot{x},\dot{x})- 15  U'x^{(4)} + 3 \binner{\dot{x}}{x^{(5)}} \big) + \O(h^6)
\end{align*}
and
\begin{align*}
\cL_\mesh([x],h) 
&= \frac{1}{2} |\dot{x}|^2 - U + \frac{h^2}{24} \left(-2 U''(\dot{x},\dot{x}) - 2 U'\ddot{x} - |\ddot{x}|^2 \right) \\
&\quad + \frac{h^4}{720} \Big( 3 U''(\ddot{x},\ddot{x}) + 6 U^{(3)}(\ddot{x},\dot{x},\dot{x}) + 4 U''(x^{(3)},\dot{x}) + 2 \big|x^{(3)}\big|^2 \\
&\hspace{4cm}  + U^{(4)}(\dot{x},\dot{x},\dot{x},\dot{x})+ U' x^{(4)} + \binner{\ddot{x}}{x^{(4)}} \Big) + \O(h^6).
\end{align*}
To eliminate higher derivatives of $x$ in the $h^4$-term we can use 
$ \ddot{x} = - U' + \O(h^2) $ as before. To do this in the $h^2$-term, the second order term of the modified equation \eqref{mech-modeqn} is also necessary. We apply it repeatedly until all higher derivatives are eliminated. We find
\begin{align*}
\cL_{\mod,5}(x,\dot{x},h) 
&= \frac{1}{2} |\dot{x}|^2 - U + \frac{h^2}{24} \left( U'U' -2 U''(\dot{x},\dot{x}) \right) \\
&\quad + \frac{h^4}{720} \big( 3 U''(U',U') - 6 U^{(3)}(U',\dot{x},\dot{x}) - 2 U''(U''\dot{x},\dot{x}) + U^{(4)}(\dot{x},\dot{x},\dot{x},\dot{x}) \big) .
\end{align*}

\begin{remark}
The derivatives of $U$ should be considered as covariant, contravariant, or mixed tensors depending on the context. For example:
\begin{align*}
&U^{(3)}(U',\dot{x},\dot{x}) = \sum_{i,j,k} \frac{\partial^3 U}{\partial x_i \partial x_j \partial x_k} \der{U}{x_i} \dot{x}_j \dot{x}_k, \\
& U''(U''\dot{x},\dot{x}) =  \sum_{i,j} \frac{\partial^2 U}{\partial x_i \partial x_j} \left( \sum_k \frac{\partial^2 U}{\partial x_i \partial x_k} \dot{x}_k \right) \dot{x}_j.
\end{align*}
\end{remark}

The fourth truncation of the modified equation is most easily found from the meshed modified Lagrangian. We have
\begin{align*}
\der{\cL_\mesh}{x} - \frac{\d}{\d t} \der{\cL_\mesh}{\dot{x}}
&= - \ddot{x} - U' + \frac{h^2}{12} \left(U''\ddot{x} + U^{(3)}(\dot{x},\dot{x}) \right) \\
&\quad + \frac{h^4}{240} \Big( - 6 U^{(4)}(\ddot{x},\dot{x},\dot{x}) - 3 U^{(3)}(\ddot{x},\ddot{x}) \\
&\hspace{1.7cm} - 4 U^{(3)}(x^{(3)},\dot{x}) - U''x^{(4)} - U^{(5)}(\dot{x},\dot{x},\dot{x},\dot{x}) \Big) + \O(h^6).
\end{align*}
Equating this to zero and solving for $\ddot{x}$ we obtain the modified equation,
\begin{align*}
\ddot{x} &= - U' + \frac{h^2}{12} \left(U^{(3)}(\dot{x},\dot{x}) - U''U' \right) \\
&\hspace{1.2cm} + \frac{h^4}{720} \Big( 20 U^{(3)}(U''\dot{x},\dot{x}) - 8 U''(U''U') + 18 U^{(4)}(U',\dot{x},\dot{x}) \\
&\hspace{4.7cm} - 9 U^{(3)}(U',U') -3 U^{(5)}(\dot{x},\dot{x},\dot{x},\dot{x}) \Big) + \O(h^6).
\end{align*}

\subsection{Comparison with the modified Hamiltonian}

Here we consider the symplectic Euler discretization of a mechanical Lagrangian. Its discrete Lagrangian is given by Example \ref{ex-methods}$(\ref{Lh-SE1})$,
\[ L_\disc(x_j,x_{j+1},h) = \frac{1}{2} \left| \frac{x_{j+1}-x_j}{h} \right|^2 - U(x_{j}). \]
The discrete Euler-Lagrange equation is
\[ \frac{x_{j+1} - 2 x_j + x_{j-1}}{h^2} = - U'(x_j). \]
Since we are dealing with a separable continuous Lagrangian, this is the same difference equation as the one obtained by the St\"ormer-Verlet method. 

For this discretization we have
\[ \cL_\disc([x],h) = \frac{1}{2} |\dot{x}|^2 - U + \frac{h}{2} U' \dot{x} + \frac{h^2}{24} \left( \binner{\dot{x}}{x^{(3)}} - 3 U'\ddot{x} - 3 U''(\dot{x},\dot{x}) \right) + \O(h^3) \]
and
\[ \cL_\mesh([x],h) = \frac{1}{2} |\dot{x}|^2 - U + \frac{h}{2} U' \dot{x} - \frac{h^2}{24} \left( |\ddot{x}|^2 + 2 U'\ddot{x} + 2 U''(\dot{x},\dot{x}) \right) + \O(h^3). \]
The second truncation of the modified Lagrangian is
\begin{equation}\label{se-mod-lag}
\cL_{\mod,2}(x,\dot{x},h) =  \frac{1}{2} |\dot{x}|^2 - U + \frac{h}{2} U' \dot{x} + \frac{h^2}{24} \left( U'U' - 2 U''(\dot{x},\dot{x}) \right).
\end{equation}
Note that the first order term $\frac{h}{2} U' \dot{x} = \frac{h}{2}\frac{\d U}{\d t}$ does not contribute to the Euler-Lagrange equations, hence this Lagrangian is equivalent to the corresponding modified Lagrangian \eqref{SV-modlag} of the St\"ormer-Verlet method. 

We compare this to the symplectic Euler discretization of the Hamiltonian system with Hamiltonian $\cH(x,p) = \frac{1}{2} p^2 + U(x)$. The modified Hamiltonian for this system, truncated after order 2, is
\begin{align}
\cH_{\mod,2}(x,p,h) &=  \cH - \frac{h}{2} \cH_x \cH_p + \frac{h^2}{12} \big( \cH_{pp} (\cH_x,\cH_x) + \cH_{xx} (\cH_p,\cH_p) + 4 \cH_{px}(\cH_x,\cH_p) \big) \notag \\
&= \frac{1}{2}p^2 + U - \frac{h}{2} U' p + \frac{h^2}{12} \big( U'U' + U'' (p,p) \big)  . \label{se-mod-ham}
\end{align}
Its derivation can be found for example in \cite[Example IX.3.4]{hairer2006geometric}. 

Now we take the Legendre transformation of the modified Lagrangian \eqref{se-mod-lag}. We have
\[ p = \der{\cL_{\mod,2}}{\dot{x}} = \dot{x} + \frac{h}{2}U' - \frac{h^2}{6} U'' \dot{x} \]
and hence
\[ \dot{x} = p - \frac{h}{2}U' + \frac{h^2}{6} U'' p  + \O(h^3). \]
The Hamiltonian corresponding to $\cL_{\mod,2}$ is
\[ \big(\!\inner{p}{\dot{x}} - \cL_{\mod,2} \big)\big|_{\dot{x} = p - \frac{h}{2}U' + \frac{h^2}{6} U'' p + \O(h^3)} = \cH_{\mod,2} + \O(h^3). \]
We see that, up to a truncation error, the modified Lagrangian \eqref{se-mod-lag} and the modified Hamiltonian \eqref{se-mod-ham} are obtained from one another by Legendre transformation.

\subsection{A non-separable Lagrangian}

Our approach is not limited to separable Lagrangians. It can be applied whenever the Lagrangian is regular. 

\begin{example}[Anisotropic harmonic oscillator]
Consider a Lagrangian of the form 
\[ \cL(x, \dot{x}) = \frac{1}{2}\inner{ \dot{x} }{ M \dot{x} } + \frac{1}{2}\inner{ x }{ (J_+ + J_-) \dot{x} } + \frac{1}{2}\inner{ x }{ A x }, \]
where the matrices $A$, $J_+$, and $M$ are symmetric, and $J_-$ is antisymmetric. Its Euler-Lagrange equation is
\[ -M \ddot{x} + J_- \dot{x} + A x = 0. \]
We use the discrete Lagrangian from Example \ref{ex-methods}$(\ref{Lh-SE2}$),
\begin{align*}
L_\disc(x_j, x_{j+1},h) &= \cL\!\left( x_{j+1}, \frac{x_{j+1}-x_j}{h} \right) \\
&= \frac{1}{2} \inner{\frac{x_{j+1}-x_j}{h} }{ M \frac{x_{j+1}-x_j}{h}} \\
&\qquad + \frac{1}{2} \inner{ x_{j+1} }{ (J_+ + J_-) \frac{x_{j+1}-x_j}{h} }  + \frac{1}{2} \inner{ x_{j+1} }{ A x_{j+1} }.
\end{align*}
Its discrete Euler-Lagrange equation is
\[ \left( M + \frac{h}{2} J_+ \right) \frac{-x_{j+1} + 2 x_j - x_{j-1}}{h^2} + J_- \frac{ x_{j+1} - x_{j-1} }{ 2h } + Ax_j = 0.\]
Note that this depends on $J_+$, even though the continuous Euler-Lagrange equation does not. We have
\begin{align*}
\cL_\disc([x],h) &= \cL\!\left( x + \frac{h}{2} \dot{x},\  \dot{x} \right) +\O(h^2) \\
&= \cL + \frac{h}{2} \inner{ \der{\cL}{x} }{ \dot{x} } +\O(h^2) \\
&= \frac{1}{2}\inner{ \dot{x} }{ M \dot{x} } + \frac{1}{2}\inner{ x }{ (J_+ + J_-) \dot{x} } + \frac{1}{2}\inner{ x }{ A x } \\
&\hspace{3cm} + \frac{h}{2} \inner{ \frac{1}{2} (J_+ + J_-) \dot{x} + A x}{ \dot{x} } + \O(h^2).
\end{align*}
Up to first order, the meshed modified Lagrangian is equal to $\cL_\disc$,
\begin{align*}
\cL_\mesh([x],h)
&= \frac{1}{2}\inner{ \dot{x} }{ M \dot{x} } + \frac{1}{2}\inner{ x }{ (J_+ + J_-) \dot{x} } + \frac{1}{2}\inner{ x }{ A x } \\
&\hspace{3cm} + \frac{h}{2} \left( \frac{1}{2} \inner{ \dot{x} }{ J_+ \dot{x} } + \inner{ \dot{x} }{ A x } \right) + \O(h^2).
\end{align*}
Since second and higher derivatives of $x$ do not occur in these terms, the classical modified Lagrangian $\cL_{\mod,1}(x,\dot{x},h)$ is obtained by simply truncating $\cL_\mesh([x],h)$ after the first order term. Its Euler-Lagrange equation is
\[ -M \ddot{x} + J_- \dot{x} + A x - \frac{h}{2} J_+ \ddot{x} = 0. \]
Solving for $\ddot{x}$ we find the modified equation
\[ \ddot{x} = M^{-1} (J_- \dot{x} + A x) - \frac{h}{2} M^{-1} J_+ M^{-1} (J_- \dot{x} + A x) + \O(h^2) \]
We see that the first order term of the modified equation depends on $J_+$, even though the original Euler-Lagrange equation does not. This example illustrates how different but equivalent continuous Lagrangians lead to different discretizations and different modified Lagrangians. However, the leading order term of the modified equation is the same for all of them. This term is just the original Euler-Lagrange equation.
\end{example}

\section{Conclusion and outlook}

We addressed the question whether modified equations for variational integrators are Lagrangian. In a strict sense the answer is no: truncations of the modified equations are not Euler-Lagrange equations. However, they can be turned into Euler-Lagrange equations by adding higher-order corrections, or by considering the full formal power series. We developed a method to construct modified Lagrangians, starting from the discrete Lagrangian of the numerical method. This provides a new algorithm to construct the modified equation for variational integrators.

Some goals for future research are extending this method to degenerate Lagrangians, to systems with (nonholonomic) constraints and to Lagrangian PDEs, as well as a rigorous study of the optimal truncation of the power series and possible long-time conservation results.

\bibliographystyle{plainnat}
\bibliography{modeqn}

\begin{thebibliography}{11}
\providecommand{\natexlab}[1]{#1}
\providecommand{\url}[1]{\texttt{#1}}
\expandafter\ifx\csname urlstyle\endcsname\relax
  \providecommand{\doi}[1]{doi: #1}\else
  \providecommand{\doi}{doi: \begingroup \urlstyle{rm}\Url}\fi

\bibitem[Abramowitz et~al.(1972)Abramowitz, Stegun,
  et~al.]{abramowitz1972handbook}
Milton Abramowitz, Irene~A Stegun, et~al.
\newblock \emph{Handbook of mathematical functions}, volume~1.
\newblock Dover New York, 1972.

\bibitem[Calvo et~al.(1994)Calvo, Murua, and Sanz-Serna]{calvo1994modified}
M.P. Calvo, A.~Murua, and J.M. Sanz-Serna.
\newblock Modified equations for {ODE}s.
\newblock \emph{Contemporary Mathematics}, 172:\penalty0 63--63, 1994.

\bibitem[Hairer(1994)]{hairer1994backward}
Ernst Hairer.
\newblock Backward analysis of numerical integrators and symplectic methods.
\newblock \emph{Annals of Numerical Mathematics}, 1:\penalty0 107--132, 1994.

\bibitem[Hairer et~al.(2003)Hairer, Lubich, and Wanner]{hairer2003geometric}
Ernst Hairer, Christian Lubich, and Gerhard Wanner.
\newblock Geometric numerical integration illustrated by the
  st{\"o}rmer--verlet method.
\newblock \emph{Acta numerica}, 12:\penalty0 399--450, 2003.

\bibitem[Hairer et~al.(2006)Hairer, Lubich, and Wanner]{hairer2006geometric}
Ernst Hairer, Christian Lubich, and Gerhard Wanner.
\newblock \emph{Geometric numerical integration: structure-preserving
  algorithms for ordinary differential equations}, volume~31.
\newblock Springer, 2006.

\bibitem[Marsden and West(2001)]{marsden2001discrete}
Jerrold~E. Marsden and Matthew West.
\newblock Discrete mechanics and variational integrators.
\newblock \emph{Acta Numerica 2001}, 10:\penalty0 357--514, 2001.

\bibitem[Moan(2006)]{moan2006modified}
P.C. Moan.
\newblock On modified equations for discretizations of {ODE}s.
\newblock \emph{Journal of Physics A: Mathematical and General}, 39\penalty0
  (19):\penalty0 5545, 2006.

\bibitem[Oliver and Vasylkevych(2012)]{oliver2012new}
Marcel Oliver and Sergiy Vasylkevych.
\newblock A new construction of modified equations for variational integrators.
\newblock 2012.

\bibitem[Reich(1999)]{reich1999backward}
Sebastian Reich.
\newblock Backward error analysis for numerical integrators.
\newblock \emph{SIAM Journal on Numerical Analysis}, 36\penalty0 (5):\penalty0
  1549--1570, 1999.

\bibitem[Vermeeren(2015)]{vermeeren2015dynamical}
Mats Vermeeren.
\newblock A dynamical solution to the {B}asel problem.
\newblock \texttt{arXiv:1506.05288 [math.CA]}, 2015.

\bibitem[Vermeeren(2016)]{vermeeren2016numerical}
Mats Vermeeren.
\newblock Numerical precession in variational integrators for the {K}epler
  problem.
\newblock \texttt{arXiv:1602.01049 [math.NA]}, 2016.

\end{thebibliography}

\end{document}